\documentclass[11pt]{amsart}
\usepackage{mathrsfs}

\usepackage{latexsym}
\usepackage{amssymb}
\usepackage{amsfonts}

\usepackage{amscd}
\usepackage{bbm}
\usepackage{mathabx}
\usepackage{skak}

\usepackage{color}

\usepackage{hyperref}
\hypersetup{
    colorlinks,
    citecolor=black,
    filecolor=black,
    linkcolor=black,
    urlcolor=black
}

\usepackage[all]{xy}

\newtheorem{proposition}{Proposition}[section]

\newtheorem{lemma}[proposition]{Lemma}
\newtheorem{definition}[proposition]{Definition}
\newtheorem{theorem}[proposition]{Theorem}

\newtheorem{corollary}[proposition]{Corollary}

\newtheorem{example}[proposition]{Example}
\newtheorem{remark}[proposition]{Remark}
\newtheorem{construction}[proposition]{Construction}
\newtheorem{lemma-definition}[proposition]{Lemma-Definition}

\newcounter{tmp}

\def\lto{\longrightarrow}

\def\D{{\mathcal D}}

\def\H{{\mathcal H}}

\def\T{{\mathcal T}}
\def\I{{\mathcal I}}

\def\P{{\mathcal{P}}}

\def\ZZ{{\mathbb Z}}

\def\bR{{\mathbf R}}

\def\bL{{\mathbf L}}

\def\NN{{\mathbb N}}
\def\ZZ{{\mathbb Z}}

\def\Hom{\operatorname{Hom}}

\def\Ker{\operatorname{Ker}\,}

\def\id{{\operatorname{id}}}

\def\codim{{\operatorname{codim}\;}}

\def\kk{{\Bbbk}}

\def\op{\circ}

\newcommand{\Ho}{{\H^0}}

\newcommand{\SF}{\dS\!\dF\!\operatorname{--}\!}
\newcommand{\SFf}{\dS\!\dF_{fg}\!\operatorname{--}\!}
\newcommand{\prfdg}{\mathscr{P}\!\mathit{erf}\!\operatorname{--}}

\newcommand{\Ac}{\dA\!\mathit{c}\!\operatorname{--}\!}

\newcommand{\prf}{\operatorname{perf}\!\operatorname{--}}

\def\dA{\mathscr A}
\def\dB{\mathscr B}
\def\dC{\mathscr C}

\def\dE{\mathscr E}

\def\dF{\mathscr F}

\def\dR{\mathscr R}
\def\dS{\mathscr S}

\def\Mod{{\mathscr M}\!\mathit{od}\!\operatorname{--}\!}

\def\mE{\mathsf E}

\def\mF{\mathsf F}
\def\mI{\mathsf I}
\def\mJ{\mathsf J}

\def\mM{\mathsf M}
\def\mN{\mathsf N}
\def\mP{\mathsf P}

\def\mS{\mathsf S}
\def\mT{\mathsf T}

\def\mf{\mathsf f}

\def\mPhi{\mathsf \Phi}

\def\m0{\mathsf 0}

\def\dHom{\mathsf{Hom}}
\def\dEnd{\mathsf{End}}

\def\rd{{J}}
\def\rdi{\mJ_{-}}
\def\rde{\mJ_{+}}

{\endgroup\hfill$\Box$}

{\endgroup\hfill$\Box$}

\def\hy{\mbox{-}}

\def\gR{R}
\def\gS{S}
\def\gM{M}
\def\gN{N}
\def\gHom{\Hom}

\def\La{\Lambda}

\def\da{d_{\dA}}
\def\db{d_{\dB}}
\def\dc{d_{\dC}}

\def\dr{d_{\dR}}
\def\ds{d_{\dS}}

\def\cf{\mathrm{cf}}

\def\bone{{\mathbf 1}}
\def\btwo{{\mathbf 2}}
\def\rAnn{{\operatorname{rAnn}}}
\def\pd{{\operatorname{pd}}}

\def\Chi{\mathtt X}
\def\mod{\operatorname{mod}\!}

\def\fd{\mathsf{cf}}

\def\SL{\operatorname{SL}}

\title[]{Smooth DG algebras and twisted tensor product}

\author[]{Dmitri Orlov}

\address{ Algebraic Geometry Dept., Steklov Math. Institute RAS,
8 Gubkin str., Moscow 119991, RUSSIA}
\email{orlov@mi-ras.ru}

\thanks{This work is supported by Russian Science Foundation under  grant 19-11-00164, \href{https://rscf.ru/project/19-11-00164/}{https://rscf.ru/project/19-11-00164}}

\date{}
\dedicatory{Dedicated to the blessed memory of Igor Rostislavovich Shafarevich on the occasion of his 100th birthday}

\keywords{Noncommutative algebraic geometry, differential graded algebras,  perfect modules}

\makeatletter
\@namedef{subjclassname@2020}{\textup{2020} Mathematics Subject Classification}
\makeatother
\subjclass[2020]{14A22, 16E45, 16P10, 16E35, 18G80}

\begin{document}

\begin{abstract}
In this paper, twisted tensor product of DG algebras is studied and sufficient conditions for smoothness of such a product are given. It is shown that in the case of finite-dimensional DG algebras, applying this operation offers great possibilities for constructing new examples of smooth DG algebras and algebras. In particular, examples are given of families of algebras of finite global dimension with two simple modules that have nontrivial moduli spaces.

\end{abstract}

\maketitle


\section*{Introduction}

The main objects of study in algebraic geometry are smooth projective varieties. The property of being projective, or more precisely of being proper, can be easily seen on the level  of the category of perfect complexes on a variety. Namely, morphisms between perfect complexes in this case are finite-dimensional vector spaces. This property is generalized to noncommutative varieties and to differential graded (DG) algebras directly. Another important property is the property of smoothness. Smoothness, like the property of being regular, is also fundamental and can be extended both to noncommutative varieties and to DG algebras.

The main goal of this paper is to study  smooth finite-dimensional algebras and DG algebras. For finite-dimensional DG algebras, the property of properness is already fulfilled.
Regularity (or smoothness) for finite-dimensional algebras is in fact equivalent to the finiteness of the global dimension with the additional separability property for the semisimple part if we are talking about smoothness.  We show that these properties can also be extended to DG algebras.

A natural question arises: what kind of algebras or DG algebras of finite global dimension do we know and how to construct them?
Algebras of finite global dimension can be naturally obtained from a directed quiver $Q$ with arbitrary relations $I.$
Recall that a quiver $Q$ is said to be directed if a total order is given on its (finite) set of vertices, and for any arrow the target vertex   is strictly greater than the source vertex.
Any algebra of the form $A=\kk Q/I,$ where $\kk Q$ is the path algebra of the quiver $Q$ over a field $\kk,$ and $I$ is an ideal of relations, has  finite global dimension.
In this case, the category of perfect complexes $\prf A,$ which is equivalent to the derived category of finite-dimensional modules $\D^b(\mod-A),$ is a triangulated category with a full exceptional
collection. From the point of view of DG categories, the DG category of perfect complexes $\prfdg A$ can be obtained as a gluing of several categories of the form $\prfdg \kk$ via perfect bimodules.
Such categories always turn out to be smooth. And thus, the procedure of gluing smooth categories via perfect bimodules is the main operation for obtaining new smooth categories (see \cite{Or16,Or19,Or20}).

However, it is important to study more general algebras and DG-algebras of finite global dimension: in particular, those that cannot be obtained by the gluing procedure on the level of the category of perfect complexes.
Algebras of such  type with two simple modules appeared in \cite{G}, and it was shown  in \cite{Ha,MH} that the categories of perfect complexes for these algebras do not have full exceptional collections.
Another series of algebras of finite global dimension with two simple modules was constructed in \cite{KK}. One of the interesting features of these algebras is that their nilpotency index is equal to 4, while the global dimension can be arbitrarily large.

The question also arises of finding other operations that allow one to construct smooth algebras and DG algebras, starting from known and elementary smooth algebras. A well-known example of such an operation is  the usual tensor product. The tensor product of two smooth DG $\kk$\!--algebras
$\dA\otimes_{\kk}\dB$ is also smooth (see \cite{Lu}). However, this operation has strong limitations and naturally increases the rank of the semisimple part. Moreover, it can be defined only over a central subring.

In this paper we consider the operation of twisted tensor product of noncommutative algebras and DG algebras over various subalgebras that do not belong to the center, and apply this operation to the construction of new smooth finite-dimensional algebras and DG algebras. Note that the twisted tensor product has appeared in the literature before. For example, \cite{CSV} defined twisted tensor product over a central subring, and, in \cite{Du}, for a pair of finite-dimensional algebras with a common semisimple part,  a certain operation was introduced, which should be understood as a twisted tensor product of these algebras over a semisimple subalgebra.

In this paper, we give sufficient conditions for a (DG) twisted tensor product of DG algebras to be  smooth (see Theorems \ref{DGprod} and \ref{DGfinpr}).
In the case where these DG algebras admit augmentations, one can present an explicit twisted map for them that defines a certain twisted tensor product of these DG algebras and allows one to construct new examples of smooth DG algebras (see (\ref{vtwist}) from Construction \ref{Mtwist}). In section \ref{ExampAlg}, we consider examples of algebras of finite global dimension with two simple modules and show how they can be obtained as twisted tensor products. Furthermore, we give new examples of families of algebras with two simple modules that have finite global dimension (see Theorem \ref{fgdim}) and show that the categories of perfect modules over these algebras do not have exceptional objects (see Corollary \ref{noexc}). In contrast to already known algebras of such type, these families of algebras have nontrivial moduli spaces. We also show that all these algebras can be obtained by a sequence of twisted tensor products of elementary smooth algebras (see Theorem \ref{TwProd}). In the last section, we consider the Grothendieck groups of smooth DG algebras and show that any matrix in $\SL(n,\ZZ)$ can be realized as the matrix of the Euler bilinear form  (\ref{biform}) for some smooth DG algebra (see Corollary \ref{anymatr}).

The author is very grateful to Anton Fonarev and Alexander Kuznetsov for useful discussions and valuable comments.
\medskip

\section{Preliminaries}

\subsection{Differential graded algebras}

Let $\kk$  be a field. Recall that  a {\sf differential graded $\kk$\!--algebra (=DG
algebra)} $\dR=(\gR, \dr)$ is  a $\ZZ$\!--graded associative $\kk$\!-algebra
$
\gR =\bigoplus_{q\in \ZZ} R^q
$
endowed with a $\kk$\!-linear differential $\dr: \gR \to \gR$  (i.e. homogeneous
map $\dr$ of degree 1 with $\dr^2 = 0$) that satisfies the graded Leibniz rule
\[
\dr(xy) = \dr (x) y + (-1)^q x \dr (y) \quad \text{for all}\quad  x\in R^q, y\in \gR.
\]
We consider DG algebras with an identity element $1\in R^0.$ In this case, $\dr(1)=0.$

Notice that any ordinary associative $\kk$\!--algebra $\La$ can be considered as a DG algebra $\dR$ with
$R^0=\La$ and $R^q = 0,$ when $q \ne 0.$

A {\sf differential graded module $\mM$ over  $\dR$ (=DG $\dR$\!--module)} is a $\ZZ$\!-graded right
$\gR$\!-module
$
\gM = \bigoplus_{q\in\ZZ} M^q
$
endowed with a $\kk$\!-linear differential $d_{\mM}: \gM \to \gM$ of degree 1 for which $d_{\mM}^2=0$ satisfying
the graded Leibniz rule, i.e.
\[
d_{\mM}(mr) = d_{\mM}(m) r + (-1)^q m \dr( r) , \quad\text{for all}\quad m\in M^q, r\in \gR.
\]

Let $\mM$ and $\mN$ be two DG modules. We can define a complex of $\kk$\!--vector spaces $\dHom_{\dR} (\mM, \mN)$
as the graded vector space
\begin{equation*}
\gHom_{\gR}^{gr} (\gM, \gN):=\bigoplus_{q\in\ZZ}\Hom_{\gR} (\gM, \gN)^q,
\end{equation*}
where $\Hom_{\gR} (\gM, \gN)^q$ is the space of homogeneous homomorphisms of $\gR$\!--modules of degree $q.$
The differential $D$ of the complex $\dHom_{\dR} (\mM, \mN)$ is defined by the following rule
\begin{equation*}
D(f) = d_{\mN} \circ f - (-1)^q
f\circ d_{\mM}\quad\text{ for each}\quad f\in \Hom_{\gR} (\gM, \gN)^q.
\end{equation*}

Thus all (right) DG $\dR$\!--modules form a DG category
$\Mod \dR.$ Let $\Ac\dR$ be the full
DG subcategory consisting of all acyclic DG modules, i.e. DG modules with trivial cohomology.
The
homotopy category $\Ho(\Mod\dR)$ has a natural structure of a triangulated category,
and the homotopy subcategory of acyclic complexes $\Ho (\Ac\dR)$ forms a full triangulated subcategory in it.
The {\sf derived
category} $\D(\dR)$ is defined as the Verdier quotient
\[
\D(\dR):=\Ho(\Mod\dR)/\Ho (\Ac\dR).
\]

It is well-known that the derived category $\D(\dR)$ is equivalent to the homotopy category $\Ho(\SF\dR),$
where $\SF\dR\subset\Mod\dR$ is the DG subcategory of semi-free modules.
Recall that a DG module
$\mP$ is called {\sf semi-free} if it has a filtration
$0=\mPhi_0\subset \mPhi_1\subset ...=\mP=\bigcup \mPhi_n$
with free quotients  $\mPhi_{i+1}/\mPhi_i$ (see \cite{Ke}).
We will also need notions of semi-projective and semi-flat modules.

\begin{definition}\label{spr} A DG $\dR$\!--module $\mM$ is called {\sf semi-projective (DG projective)} if the following equivalent conditions hold:
\begin{itemize}
\item[1)] the DG functor $\dHom_{\dR}(\mM, -)$ preserves surjective quasi-isomorphisms,
\item[2)] $\mM$ is projective as an $R$\!--module and $\dHom_{\dR}(\mM, -)$ preserves quasi-isomorphisms,
\item[3)] $\mM$ is a direct summand of some semi-free DG $\dR$\!--module.
\end{itemize}
\end{definition}

\begin{definition}\label{sflat} A DG $\dR$\!--module $\mM$ is called {\sf semi-flat (DG flat)} if the following equivalent conditions holds:
\begin{itemize}
\item[1)] the DG functor $\mM\otimes_{\dR}(-)$ preserves injective quasi-isomorphisms,
\item[2)] $\mM$ is flat as an $R$\!--module and the DG functor $\mM\otimes_{\dR}(-)$ preserves quasi-isomorphisms.
\end{itemize}
\end{definition}

It is easy to see that any semi-projective module is semi-flat, and the homotopy category of semi-projective module is equivalent to the category $\Ho(\SF\dR)\cong \D(\dR).$

\subsection{Categories of perfect modules and functors}

Denote by $\SFf\dR\subset \SF\dR$ the full DG subcategory of finitely generated semi-free
DG modules, i.e. such semi-free DG modules that $\mPhi_n=\mP$ for some $n,$ and $\mPhi_{i+1}/\mPhi_i$ is a finite direct sum of
$\dR[m].$
The {\sf DG category of perfect modules} $\prfdg\dR$
is the full DG subcategory of $\SF\dR$ consisting of all DG modules that are isomorphic to direct summands of objects of $\SFf\dR$
in the homotopy category $\Ho(\SF\dR).$
The homotopy category $\Ho(\prfdg\dR),$ which we denote by $\prf\dR,$  is called the triangulated category of perfect modules.
It is equivalent to the triangulated subcategory of compact objects $\D(\dR)^c\subset \D(\dR)$ (see \cite{Ke}).
In other words, the category $\prf\dR$ is the subcategory in $\D(\dR)$  that is (classically) generated by the DG algebra $\dR$ itself in the following sense.

\begin{definition}
A set $S$ of objects  of a triangulated category $\T$ {\sf (classically) generates} $\T$
if the smallest full triangulated subcategory of
$\T$ containing $S$ and  closed under taking direct summands coincides with the whole category $\T.$
In the case where the set $S$ consists of a single object $E\in \T,$
the object $E$ is called a {\sf classical generator} of $\T.$
\end{definition}

A classical generator, which  generates a triangulated category in a finite number of steps, is called
a {\sf strong generator}.
More precisely, let $\I_1$ and $\I_2$ be two full subcategories of a triangulated category $\T.$ Denote by $\I_1*\I_2$ the full subcategory of $\T$
consisting of all objects $M$ for which there is an exact triangle $M_1\to M\to M_2$ with $M_i\in \I_i.$
For any subcategory $\I\subset\T$ denote by $\langle \I\rangle$ the smallest full subcategory of $\T$ containing $\I$ and closed under
finite direct sums, direct summands and shifts. We put $\I_1 \diamond\I_2=\langle \I_1*\I_2\rangle$ and  define by induction
$\langle \I\rangle_k=\langle\I\rangle_{k-1}\diamond\langle \I\rangle.$ If $\I$ consists of a single object $E,$ we denote $\langle \I\rangle$ as
 $\langle E\rangle_1$ and put by induction $\langle E\rangle _k=\langle E\rangle_{k-1}\diamond\langle E\rangle_1.$
\begin{definition}
An object $E\in\T$ is called  a {\sf strong generator} if $\langle E\rangle_n=\T$ for some $n\in\NN.$
\end{definition}

It is easy to see that if a triangulated category has a strong generator, then all of
its classical generators are  also strong.

Let $\dR$ and $\dS$ be two DG algebras and let  $\mf:\dR \to \dS$ be a morphism of DG algebras.
It induces the restriction DG functor
$
\mf_*:\Mod\dS\lto \Mod\dR
$
between the DG categories of DG modules.
The restriction functor $\mf_*$ has left and right adjoint functors $\mf^*, \mf^{!}$ that are defined as follows:
\[
\mf^*\mM=\mM\otimes_{\dR} \dS,\quad \mf^{!}=\dHom_{\Mod\dR}(\dS, \mM).
\]
The DG functor $\mf_*$ preserves acyclic DG modules and induces a derived functor $\bR f_*: \D(\dS)\to \D(\dR),$
and the DG functor $\mf^*$ preserves semi-free   DG modules.
Existence of semi-free resolutions allows us to a define derived functor
$\bL f^*$ from $\D(\dR)$ to $\D(\dS)$ (see \cite{Ke}). For example, the derived functor $\bL f^*: \D(\dR)\to \D(\dR)$ is isomorphic to the induced homotopy functor
$\H^0(\mf^*)$ for the extension DG functor $\mf^*: \SF\dR\to\SF\dS.$

More generally, let $\mT$ be an
$\dR\hy\dS$\!--bimodule, that is (by definition) a DG-module over $\dR^{\op}\otimes_{\kk}\dS.$
For each DG $\dR$\!--module $\mM$ we obtain a DG $\dS$\!--module
$\mM\otimes_{\dR} \mT.$
The DG functor $(-)\otimes_{\dR} \mT: \Mod\dR \to \Mod\dS$ admits a right adjoint
$\dHom_{\dS} (\mT, -).$
These functors induce an adjoint pair of  derived functors
$(-)\stackrel{\bL}{\otimes}_{\dR}\mT,$ and
$\bR \Hom_{\dS} (\mT, -)$
between the derived categories $\D(\dR)$ and
$\D(\dS)$ (see \cite{Ke}).

\subsection{Properties of DG algebras}
Let us now discuss some basic properties of DG algebras and categories of perfect modules over DG algebras.
\begin{definition}\label{propdef} Let $\dR$ be a DG $\kk$\!--algebra. Then
\begin{itemize}
\item[(1)] $\dR$ is called {\sf proper} if the cohomology algebra $\bigoplus_{p\in\ZZ}H^p(\dR)$ is finite-dimensional.
\item[(2)] $\dR$ is called {\sf $\kk$\!--smooth} if  it is perfect as a DG bimodule, i.e. as a DG module over the DG algebra $\dR^{\op}\otimes_{\kk}\dR.$
\item[(3)] $\dR$ is called {\sf regular} if $\dR$ is a strong generator for  the triangulated category $\prf\dR.$
\end{itemize}
\end{definition}

All these properties are properties of the DG category $\prfdg\dR.$ It is easy to see that $\dR$ is proper if and only if
$\bigoplus_{m\in\ZZ}\Hom(X, Y[m])$ is finite-dimensional for any two objects $X, Y\in\prf\dR.$ It was proved in \cite{LS} that  smoothness is invariant under Morita equivalence, i.e.
it is a property of the DG category $\prfdg\dR.$ It is also known that
any smooth DG algebra is regular (see \cite{Lu}).

\begin{definition}
A DG $\dR$\!--module $\mM$ is called {\sf cohomologically finite-dimensional} if it is perfect as a complex of $\kk$\!--vector spaces, i.e.
$\bigoplus_{p\in\ZZ}H^p(\mM)$ is a finite-dimensional vector space. Denote by $\D_{\fd}(\dR)\subset \D(\dR)$ the full triangulated subcategory of cohomologically finite-dimensional DG modules.
\end{definition}

It is obvious that $\prf\dR\subseteq \D_{\fd}(\dR)$ for any proper DG algebra $\dR.$ On the other hand, if $\dR$ is smooth, then there is the opposite inclusion $\D_{\fd}(\dR)\subseteq \prf\dR$
(see, e.g., \cite{KS}). Thus, for any smooth and proper DG algebra $\dR$ we have an equivalence $\D_{\fd}(\dR)\cong \prf\dR.$  Moreover, the following proposition
is a particular case of Theorem 1.3 from \cite[Th. 1.3]{BV}.
\begin{proposition}\cite{BV}\label{BonVdB}
Let  $\dR$ be a proper and regular DG algebra. Then $\D_{\fd}(\dR)\cong\prf\dR.$
\end{proposition}

However, it should be noted that regularity itself, in contrast to smoothness,  does not imply that any cohomologically finite-dimensional module is perfect.
The simplest example is a gluing of the field $\kk$ with itself via an infinite-dimensional vector space.

\begin{proposition} Let $\dR$ be a DG $\kk$\!--algebra. Then the following conditions are equivalent:
\begin{itemize}
\item[(1)] the DG algebra $\dR$ is smooth and proper;
\item[(2)] the DG algebra $\dR^{\op}\otimes_{\kk} \dR$ is smooth and proper;
\item[(3)] there is an equivalence $\prf(\dR^{\op}\otimes_{\kk}\dR)\cong \D_{\fd}(\dR^{\op}\otimes_{\kk}\dR).$
\end{itemize}
\end{proposition}
\begin{proof} If $\dR$ is smooth, then $\dR^{\op}$ and $\dR^{\op}\otimes_{\kk}\dR$ are smooth too (see, e.g., \cite[Lem 3.3]{Lu}). Hence $(1)\Rightarrow (2).$
Further, $(2)\Rightarrow (3)$ by Proposition \ref{BonVdB}, because smoothness of $\dR$ implies regularity by \cite[Lem 3.6]{Lu}.
If $(3)$ holds, then $\dR$ is proper and it is perfect as a bimodule. Therefore, $\dR$ is smooth. Thus $(3)\Rightarrow (1).$
\end{proof}

\begin{definition} Let $\dR$ be a proper DG algebra. We say that $\dR$ has a regular (or smooth) realization if there is a proper and regular (or smooth) DG algebra $\dS$ and
an $\dR\hy\dS$\!--bimodule $\mT$ such that the functor
$
\bL F^*\cong (-)\stackrel{\bL}{\otimes}_\dR \mT: \prf\dR\to \prf\dS
$
is  fully faithful.
\end{definition}
\begin{remark}{\rm
Note that there is an example of a proper DG algebra that does not have a smooth realization. Such an example can be found in Efimov's paper \cite{Ef}
(Theorem 5.4 and Proposition 5.1).
}
\end{remark}

\begin{proposition}\label{regcat} Let $\dR$ be a proper DG algebra.
Assume that $\dR$ has a regular (or smooth) realization.
Then $\dR$ is regular (or smooth) if and only if $\prf\dR\cong \D_{\fd}(\dR).$
\end{proposition}
\begin{proof} In one direction, this  follows from  Proposition \ref{BonVdB}.

Assume now that $\prf\dR\cong \D_{\fd}(\dR),$ and consider a regular (or smooth) realization
\[
\bL F^*\cong (-)\stackrel{\bL}{\otimes}_\dR \mT: \prf\dR\hookrightarrow \prf\dS.
\]
Since the DG algebra $\dS$ is regular (or smooth), we also have an equivalence
$\prf\dS\cong \D_{\fd}(\dS).$ We know that $\mT=\bL F^*(\dR)\in \prf\dS.$ Hence, the  right adjoint functor
$
\bR \Hom_{\dS} (\mT, -): \D(\dS)\to \D(\dR)
$
sends $\D_{\fd}(\dS)$ to $\D_{\fd}(\dR).$ Therefore, we obtain a projection
\[
\bR \Hom_{\dS} (\mT, -): \prf\dS\to \prf\dR.
\]
If now $\dS$ is regular, then $\dS$ is a strong generator for $\prf\dS.$ Hence, the object $\bR \Hom_{\dS} (\mT, \dS)$ is a strong generator
for $\prf\dR,$ and $\dR$ is regular too. If, in addition, $\dS$ is smooth, then $\dR$ is also smooth by \cite[3.24]{LS}.
\end{proof}

\section{Morphisms of DG algebras and twisted tensor product}

\subsection{Morphisms of DG algebra and smoothness}
Let now $\mf:\dR \to \dS$ be a morphism of DG algebras. As above, it produces the derived functors
\[
\bL \mf^*: \D(\dR)\to \D(\dS)
\quad
\text{and}
\quad
\bR\mf_*: \D(\dS)\to \D(\dR),
\]
 which are called the inverse image and the direct image functors, respectively.
These derived functors induce functors
\[
\bL \mf^*: \prf\dR\to \prf\dS
\quad
\text{and}
\quad
\bR\mf_*: \D_{\fd}(\dS)\to \D_{\fd}(\dR).
\]

\begin{definition}\cite[Def.2.8]{Or19}
 A morphism of DG algebras $\mf:\dR \to \dS$   is called a {\sf pp-morphism (perfect proper morphism)}
if the direct image functor $\bR\mf_*: \D(\dS)\to \D(\dR)$ sends perfect modules to perfect ones.
\end{definition}
The previous definition is equivalent to saying that $\dS$ is  perfect as a right DG $\dR$\!--module.
This also means that the inverse image functor $\bL \mf^*$ considered as a functor from $\prf\dR$ to $\prf\dS$ has a right adjoint
$\bR \mf_*: \prf\dS\to\prf\dR.$

\begin{definition} Let $\dR$ be a proper DG algebra.
 A morphism of DG algebras $\mf:\dR \to \dS$  will be called  {\sf respectable}
if the objects of the form $\bR\mf_* \mT$ for $\mT\in\D_{\fd}(\dS),$ generate the whole triangulated category $\D_{\fd}(\dR).$
The morphism $\mf$ will be called {\sf acceptable}, if the subcategory $\D\subseteq\D_{\fd}(\dR),$  generated by the objects $\bR\mf_* \mT$ for $\mT\in\D_{\fd}(\dS),$
contains the category $\prf\dR.$
\end{definition}

The next lemma is almost obvious.
\begin{lemma}\label{ppr} Let $\mf:\dR \to \dS$ be a respectable pp-morphism of DG algebras.
Assume that there is an inclusion  $\D_{\fd}(\dS)\subseteq\prf\dS.$
Then there is an equivalence
$\D_{\fd}(\dR)\cong\prf\dR.$
\end{lemma}
\begin{proof} Since $\mf$ is a pp-morphism, we have a functor $\bR \mf_*: \prf\dS\to\prf\dR.$ If there is an inclusion $\D_{\fd}(\dS)\subseteq\prf\dS,$ then for any object
$\mT\subset \D_{\fd}(\dS),$ the image $\bR \mf_*\mT$  belongs to $\prf\dR.$ By assumption, the morphism $\mf$ is also respectable, i.e. the category $\D_{\fd}(\dR)$ is generated by objects of the form $\bR\mf_* \mT.$
Therefore, we obtain an inclusion $\D_{\fd}(\dR)\subseteq\prf\dR.$ On the other hand, $\dR$ is proper. Thus, we obtain an equivalence $\prf\dR\cong\D_{\fd}(\dR).$
\end{proof}

Consider a commutative diagram of morphisms of DG algebras
\begin{equation}\label{fiber}
\begin{split}
\xymatrix{
\dC \ar[r]^{p_B} & \dB\\
\dA\ar[r]^{\pi_A}\ar[u]^{i_A}& \dR\ar[u]_{\epsilon_B}
}
\end{split}
\end{equation}
satisfying the following two conditions:

\smallskip
\begin{tabular}{ll}
(Res) &
\begin{tabular}{ll}
1. & The DG algebra $\dC$ is semi-flat as a left DG $\dA$\!--module (see Definition \ref{sflat}).\\
2. & The canonical map $\dR\otimes_{\dA}\dC\to\dB$ is a quasi-isomorphism.
\end{tabular}
\end{tabular}
\medskip

The following proposition allows us to deduce smoothness of the DG algebra $\dC$ in diagram (\ref{fiber}).
\begin{proposition}\label{regprop} Let {\rm (\ref{fiber})}  be a commutative diagram of morphisms of proper DG algebras that satisfies conditions {\rm (Res)}.
Suppose the following conditions hold:
\begin{itemize}
\item[(i)] the morphism $\pi_A: \dA\to\dR$ is a pp-morphism,
\item[(ii)] the morphism $p_B: \dC\to\dB$ is respectable,
\item [(iii)] the DG algebra $\dB$ is regular (or smooth).
\end{itemize}
 Then there is an equivalence $\D_{\fd}(\dC)\cong\prf\dC.$ If, in addition, $\dC$ has a regular or (smooth) realization, then $\dC$ is also regular (or smooth).
\end{proposition}
\begin{proof}
We know that $\dR$ belongs to $\prf\dA$ as a right DG $\dA$\!--module because $\pi_A$ is a pp-morphism. Since $\dC$ is semi-flat as a left DG $\dA$\!--module, we obtain that
$\dB\cong \dR\otimes_{\dA}\dC$ is also perfect as a right DG $\dC$\!--module. Hence, $p_B$ is a pp-morphism. Applying Lemma \ref{ppr} to the morphism $p_B,$ we obtain an equivalence
$\D_{\fd}(\dC)\cong\prf\dC.$ If the DG algebra $\dC$ has a regular (or smooth) realization, then, by Proposition \ref{regcat}, it is regular (or smooth) itself.
\end{proof}

\begin{corollary}\label{smoothd}
Let {\rm (\ref{fiber})}  be a commutative diagram of morphisms of proper DG algebras that satisfies conditions {\rm (Res)}.
Suppose the DG algebras $\dA$ and $\dB$ are regular (or smooth) and the morphism $p_B$ is respectable.
Then there is an equivalence $\D_{\fd}(\dC)\cong\prf\dC.$ If, in addition, $\dC$ has a regular or (smooth) realization, then $\dC$ is also regular (or smooth).
\end{corollary}
\begin{proof}
Since $\dA$ is regular (or smooth) and $\dR$ is proper, the DG module $\dR$ belongs to $\prf\dA$ by Proposition \ref{BonVdB}. Thus, $\pi_A$ is a pp-morphism and
 Proposition \ref{regprop} implies the corollary.
\end{proof}

\subsection{Twisted tensor products of algebras}
In this section we consider and study so called twisted tensor products of algebras and DG algebras. Twisted tensor products of algebras over a central subring  were defined in \cite{CSV}.
Some special examples of twisted tensor products of finite-dimensional algebras over a common semisimple part appeared in \cite{Du} (see also \cite{Or21} for DG algebras).
We consider twisted tensor products of noncommutative (DG) algebras over an arbitrary (not necessary commutative) (DG) ring.

Let $R, A, B$ be  $\kk$\!--algebras and  $\epsilon_A: R\to A$ and $\epsilon_B: R\to B$ be morphisms of algebras.
We will say that $A$ and $B$ are $R$\!--rings (or rings over $R$).

\begin{definition}
A {\sf twisted tensor product over $R$} of two $R$\!--rings $A$ and $B$
is an $R$\!--ring $C$ together with two $R$\!--rings morphisms $i_A :
A\to C$ and $i_B : B \to C$ such that the canonical  map $\phi: A\otimes_R B \to C$
defined by $\phi(a\otimes b) := i_A(a)\cdot i_B(b)$ is an isomorphism of $R$\!--bimodules.
\end{definition}

There is a direct way to describe  twisted tensor products.
Let $\phi : A\otimes_R B \to C$ be the canonical isomorphism used in the definition
of the twisted tensor product. Then, we can define $\tau : B\otimes_R A \to A\otimes_R B$ by the rule
$\tau (b \otimes a) :=
\phi^{-1}(i_B(b) \cdot i_A(a)).
$

Conversely, let $\tau: B\otimes_R A\to A\otimes_R B$ be an $R$\!--bilinear map for which
\begin{equation}\label{fixsides}
\tau (1\otimes a)= a\otimes 1, \tau(b\otimes 1)=1\otimes b.
\end{equation}
In this case,  we can  define a multiplication $\mu_{\tau}:=(\mu_A\otimes \mu_B)\circ (1\otimes\tau\otimes 1)$ on the $R$\!--bimodule $A\otimes_R B.$
The multiplication $\mu_{\tau}$ is associative if and only if there is an equality
\begin{equation}\label{twist}
\tau\circ (\mu_B\otimes \mu_A) = (\mu_A\otimes \mu_B) \circ (1\otimes\tau\otimes 1) \circ (\tau \otimes \tau ) \circ (1\otimes\tau\otimes 1)
\end{equation}
of maps from $B\otimes_R B \otimes_R  A \otimes_R A$ to $A \otimes_R B$ (see, e.g., \cite{CSV}).
This means that the diagram
\begin{equation}\label{twisttau}
\begin{split}
\xymatrixcolsep{10pc}\xymatrix{
B\otimes_R B \otimes_R  A \otimes_R A \ar[r]^{(1\otimes\tau\otimes 1) \circ (\tau \otimes \tau ) \circ (1\otimes\tau\otimes 1)}\ar[d]_{\mu_B\otimes \mu_A} & A \otimes_R A\otimes_R B\otimes_R B\ar[d]^{\mu_A\otimes \mu_B}\\
B \otimes_R  A \ar[r]^{\tau} & A\otimes_R B
}
\end{split}
\end{equation}
should  be commutative.

\begin{definition}
An $\dR$\!--bilinear map  $\tau$ that satisfies conditions (\ref{fixsides}) and (\ref{twist}) is called a {\sf twisting map} for $A$ and $B$ over $R,$
and we denote the $R$\!--ring $(A\otimes_R B, \mu_{\tau})$ by $A\otimes_R^{\tau}B.$
\end{definition}

\begin{example}
{\rm Any cyclic division algebra $D_{\zeta}(a, b)= \kk\{x, y\}/\langle x^n-a, y^n-b, yx-\zeta xy\rangle,$ where $a, b, \zeta\in \kk^*$ and $\zeta^n=1,$ can be represented as a twisted tensor product
$\kk(\sqrt[n]{a})\otimes_{\kk}^{\tau}\kk(\sqrt[n]{b})$ with $\tau(y^k\otimes x^l)=\zeta^{kl}\cdot x^l\otimes y^k.$
}
\end{example}
\begin{example}
{\rm
The matrix algebra $M(n, \kk)$ can be represented as a twisted tensor product $(\kk[x]/x^n)\otimes_{\kk}^{\tau} (\kk[y]/y^n)$ for an appropriate twisting map $\tau.$
}
\end{example}

Suppose that the $R$\!--ring $A$ has an $R$\!--augmentation, i.e. a morphism $\pi_A: A\to R$ such that $\pi_A\circ\epsilon_A$ is the identity map.
Denote by $I_A=\Ker\pi_{A} \subset A$ the augmentation ideal.
\begin{definition} Suppose $A$ has an augmentation. A twisted tensor product $A\otimes_R^{\tau}B$ will be called {\sf right fixed (with respect to $\pi_A$)} if $B$ is flat as a left $R$\!--module and
the map $p_B: A\otimes_R^{\tau}B\to B,$  induced by the augmentation $\pi_A: A\to R,$ is a morphism of rings.
\end{definition}

\begin{remark}\label{twistedIdeal}
{\rm
In this case $p_B$ is surjective and the tensor product $I_A\otimes_R B$ is a two-sided ideal as the kernel of the morphism $p_B.$ Therefore, the twisting map $\tau$ should send $B\otimes_R I_A$ to $I_A\otimes_R B.$
}
\end{remark}
\begin{construction}\label{Mtwist}
{\rm Now we give a main example of a twisting map, which we will  use in sequel. Suppose  that both  $R$\!--rings $A$ and $B$ have $R$\!--augmentations
$\pi_A: A\to R$ and $\pi_B: B\to R.$
In such case there is a special twisting map $\mathbf{v}: B\otimes_R A\to A\otimes_R B$ given by the following rule
\begin{equation}\label{vtwist}
\mathbf{v}(b\otimes a)=\epsilon_A(\pi_B(b))\cdot a\otimes 1 + 1\otimes b\cdot \epsilon_B(\pi_A(a))-\epsilon_A(\pi_B(b))\otimes\epsilon_B(\pi_A(a)).
\end{equation}
For this twisted tensor product, we have $(1\otimes b)(a\otimes 1)=\mathbf{v}(b\otimes a)=0,$ whenever $a\in I_A, b\in I_B.$
}
\end{construction}

\subsection{Twisted tensor product of DG algebras}

The notion of a twisted tensor product can be easily extended to the case of DG algebras.
Let $\dR, \dA, \dB$ be  DG $\kk$\!--algebras and  $\epsilon_A: \dR\to \dA$ and $\epsilon_B: \dR\to \dB$ be morphisms of DG algebras.
In such case we will say that the DG algebras $\dA$ and $\dB$ are DG $\dR$\!--rings (or DG rings over $\dR$).

\begin{definition}\label{TwistDGal}
A {\sf twisted tensor product over $\dR$} of two DG $\dR$\!--rings $\dA$ and $\dB$
is a DG $\dR$\!--ring $\dC$ together with two $\dR$\!--rings morphisms $i_A :
\dA\to \dC$ and $i_B : \dB \to \dC$ such that the canonical  map $\phi: \dA\otimes_{\dR} \dB \to \dC$
defined by $\phi(a\otimes b) := i_A(a)\cdot i_B(b)$ is an isomorphism.
\end{definition}

It follows from the definition that the differential of $\dC$ is uniquely determined by the Leibniz rule, because we have
$\dc(a\otimes b)=\da(a)\otimes b+(-1)^{\deg(a)} a\otimes \db(b).$
The twisting map $\tau$ associated with a twisted tensor product of DG rings satisfies conditions (\ref{fixsides}) and (\ref{twist}) and, additionally,
\begin{equation}
\tau(\db(b)\otimes a) + (-1)^{\deg(b)}\tau( b\otimes \db(a)) = \dc( \tau(b\otimes a)),
\end{equation}
which means that the map $\tau: \dB\otimes_{\dR} \dA\to \dA\otimes_{\dR} \dB$ should be a map of DG $\dR$\!--bimodules.

Suppose that the $\dR$\!--ring $\dA$ has an $\dR$\!--augmentation, i.e. a morphism $\pi_A: \dA\to \dR$ such that the composition $\pi_A\circ\epsilon_A$ is the identity map.

\begin{definition} A twisted tensor product $\dA\otimes_{\dR}^{\tau}\dB$ will be called {\sf right fixed (with respect of $\pi_A$)} if $\dB$ is semi-flat as a left $\dR$\!--module and
the natural map $p_B: \dA\otimes_{\dR}^{\tau}\dB\to \dB,$ which  is induced by  $\pi_A: \dA\to \dR,$ is a morphism of DG algebras.
\end{definition}

When we have a right fixed twisted tensor product $\dA\otimes_{\dR}^{\tau}\dB,$ we can also consider a deformation of the differential that preserves the structure morphism
$i_A: \dA\to \dA\otimes_{\dR}^{\tau}\dB$ and the projection $p_B: \dA\otimes_{\dR}^{\tau}\dB\to \dB.$

Let  $\dA$ and $\dB$ be  $\dR$\!--rings such that $\dA$ has an augmentation $\pi_A: \dA\to \dR$ and $\dB$ is semi-flat as the left DG $\dR$\!--module.
Consider underlying algebras $A, B, R$ and let $\tau: B\otimes_{R} A\to A\otimes_{R} B$ be a twisting map such that the twisted tensor product of algebras $A\otimes_{R}^{\tau}B$ is  right fixed.
\begin{definition}
We define a {\sf DG twisted tensor product} $\dC^{\nabla}=\dA\otimes_{\dR}^{\nabla, \tau}\dB$ as a right fixed twisted tensor product of algebras $A\otimes_R^{\tau} B$
with a new differential $d_{\dC^{\nabla}}$ such that $i_A: \dA\to \dC^{\nabla}$ and  $p_B: \dC^{\nabla}\to \dB,$ induced by  $\pi_A: \dA\to \dR,$ are morphisms of DG algebras.
\end{definition}

It follows from the definition that the differential $d_{\dC^{\nabla}}$ has the following properties:
\[
d_{\dC^{\nabla}}(a\otimes 1)=\da(a)\otimes 1 \quad\text{ and}\quad d_{\dC^{\nabla}}(a\otimes b)= \da(a)\otimes b + (-1)^{\deg a} (a\otimes \db(b) + a\cdot\nabla(1\otimes b)),
\]
where $\nabla(1\otimes b) \in I_A\otimes_R B$ and $I_A\subset A$ is the augmentation ideal.

\begin{theorem}\label{DGprod}
Let $\dR$ be a proper DG algebra. Let $\dA$ and $\dB$ be proper DG $\dR$\!--rings  such that $\dA$ has an augmentation $\pi_A: \dA\to \dR$ and $\dB$ is semi-flat as a left DG $\dR$\!--module.
Assume that  the augmentation $\pi_A: \dA\to \dR$ is a pp-morphism, $\dB$ is regular (or smooth). Let $\dC^{\nabla}=\dA\otimes_{\dR}^{\nabla, \tau}\dB$ be a DG twisted tensor product such that the morphism $p_B$ is respectable.
Then, there is an equivalence $\D_{\fd}(\dC^{\nabla})\cong\prf\dC^{\nabla}.$
If, in addition, the DG algebra $\dC^{\nabla}$ has a regular or (smooth) realization, then $\dC^{\nabla}$ is also regular (or smooth).
\end{theorem}
\begin{proof} It directly follows from Proposition \ref{regprop}, because $\dC^{\nabla}$ is semi-flat as the left DG $\dA$\!--module and
$\dR\otimes_{\dA}\dC^{\nabla}\cong  \dR\otimes_{\dA}\dA\otimes_{\dR}\dB\cong \dB.$ Thus, the DG algebras $\dR, \dA, \dB, \dC^{\nabla}$ form the commutative diagram (\ref{fiber}) that satisfies properties
(Res).
\end{proof}

\section{Finite-dimensional DG algebras and twisted tensor product}

\subsection{Finite-dimensional DG algebras}
Let $\dR=(\gR, \dr)$ be a finite-dimensional DG algebra over a base field $\kk.$
Denote by $\rd\subset R$ the (Jacobson) radical of the $\kk$\!--algebra $R.$
The ideal $\rd$ is  graded.
On the other hand,  the radical $\rd\subset R$ is not necessary a DG ideal, in general. In other words, $\dr(\rd)$ is not necessary a subspace of $\rd.$
In fact, with any two-sided graded ideal $I\subset\gR$ we can associate two DG ideals $\mI_{-}$ and $\mI_{+}$ (see \cite{Or20}).

\begin{definition}\label{intext}
Let $\dR=(\gR, \dr)$ be a finite-dimensional DG algebra and $I\subset\gR$ be a graded (two-sided) ideal.
The {\sf internal} DG ideal $\mI_{-}=(I_{-}, \dr)$ consists of all $r\in I$ such that $\dr( r)\in I,$ while
the {\sf external} DG ideal  $\mI_{+}=(I_{+}, \dr)$  is the sum $I+\dr (I).$
\end{definition}

It is easy to see that $I_{-}$ and $I_{+}$ are indeed two-sided graded ideals of $\gR.$ It is evident that they are closed under the action of the differential $\dr.$
Thus, for any (two-sided) ideal $I\subset\gR$ we obtain two DG ideals  $\mI_{-}$ and $\mI_{+}$ in the DG algebra $\dR.$
If the ideal $I\subset\gR$ is closed under the action of the differential $\dr,$ then the  DG ideals $\mI_{-}$ and $\mI_{+}$
coincide with $I.$

\begin{lemma}\label{quasi-is}
The natural morphism of the DG ideals $\mI_{-}\to\mI_{+}$ is a quasi-isomorphism, and, hence, the morphism of the DG algebras
$\dR/\mI_{-}\to \dR/\mI_{+}$ is a quasi-isomorphism too.
\end{lemma}
\begin{proof}
We have to check that the complex $\mI_{+}/\mI_{-}$ is acyclic.
Let $x\in\mI_{+}$ be an element such that  $\dr (x)=w$ with $w\in\mI_{-}.$ We know that $x=y+\dr (z),$ where $y, z\in I.$
Since $\dr (y)=w,$ then, by the definition of $\mI_{-},$ we have $y\in\mI_{-}.$ Therefore, we obtain an equality $\bar{x}=\dr (\bar{z})$ in $\mI_{+}/\mI_{-},$
where $\bar{x}, \bar{z}$ are the images of the elements $x, z$ in the quotient $\mI_{+}/\mI_{-}.$
\end{proof}

\begin{definition}
Let $\dR=(\gR, \dr)$ be a finite-dimensional DG algebra and
$\rd\subset\gR$ be the radical. The DG ideals $\rdi, \rde$ will be called the {\sf internal} and {\sf external} DG radicals of the DG algebra $\dR.$
\end{definition}

Lemma \ref{quasi-is} implies that a priori different  DG ideals $\rdi$ and $\rde$ give us quasi-isomorphic DG algebras $\dR/\mJ_{-}$ and $\dR/\mJ_{+}.$
A useful property of the DG ideal $\rdi$ is its nilpotency ( because it is  a subideal of $\rd$), while a useful property of the DG ideal $\rde$
is semisimplicity of the underlying algebra ( because it is  a quotient of the semisimple algebra $S=R/\rd$).
It is proved in \cite{Or20} that semisimplicity of the underlying algebra implies semisimplicity of a DG algebra.

Recall that a finite-dimensional DG algebra $\dS$ is called {\sf simple}
if the DG category $\prfdg\dS$ of perfect DG modules is quasi-equivalent to $\prfdg D,$ where $D$ is a finite-dimensional division
$\kk$\!--algebra, and it is called {\sf semisimple} if  the DG category $\prfdg\dS$ is quasi-equivalent to a sum $\prfdg D_1\oplus\cdots\oplus \prfdg D_m,$ where all $D_i$ are finite-dimensional division algebras over $\kk.$
In addition, $\prfdg\dS$ is called {\sf separable}, if all division algebras $D_i$ are separable over $\kk$ (see Definition 2.11 of \cite{Or20}).

\begin{proposition}\cite[Prop. 2.16]{Or20}\label{eqsemisimple}
Let $\dS=(\gS, \ds)$ be a DG algebra over a field $\kk$ such that $S$ is a semisimple algebra. Then $\dS$ is a semisimple DG algebra.
\end{proposition}

This proposition implies that both DG algebras $\dR/\mJ_{-}$ and $\dR/\mJ_{+}$ are semisimple, since they are quasi-isomorphic and
the underlying algebra for $\dR/\mJ_{+}$ is semisimple.

\subsection{Finite-dimensional DG modules}
Now let us consider the internal DG radical $\mJ_{-}.$ It is nilpotent as a subideal of the radical $J.$
Let $J\supset J^2\supset\cdots\supset J^n=0$ be the powers of the radical $J,$ where $n$ is the index of nilpotency of $J.$
Denote by $\mJ_p$ the internal DG ideals $(\rd^p)_{-}$ for $p=1,\ldots, n.$
This gives us  a chain of DG ideals  $\rdi=\mJ_{1}\supseteq\cdots\supseteq\mJ_n=0,$ and it is easy to check that
$\mJ_{p}\mJ_{q}\subseteq \mJ_{p+q}.$

\begin{lemma}\label{fdmod} Let $\dR$ be a finite-dimensional DG algebra. Then the triangulated subcategory $\D\subseteq\D_{\fd}(\dR),$  generated by the DG $\dR$\!--module $\dR/\rde,$ contains
all finite-dimensional DG modules.
\end{lemma}
\begin{proof} For any
finite-dimensional DG $\dR$\!--module $\mM$ there is  a filtration
\[
\mM=\mM\supseteq\mM\mJ_{1}\supseteq\cdots\supseteq\mM\mJ_n=0,
\]
where each quotient $\mM\mJ_{p}/\mM\mJ_{p+1}$ is a DG module over the semisimple DG algebra $\dR/\rdi.$
Thus, any quotient $\mM\mJ_{p}/\mM\mJ_{p+1}$ and, hence, any finite-dimensional DG module $\mM$ belongs to the triangulated subcategory
$\D$ that is generated by the DG $\dR$\!--module $\dR/\rdi\cong \dR/\rde.$
\end{proof}

The following Proposition is proved in \cite{Or20} (see Proposition 2.5).

\begin{proposition}\cite[Prop. 2.5]{Or20}\label{obj}
Let $\dR$ be a finite-dimensional DG algebra. Let $\mM$ be a perfect DG $\dR$\!--module.
Then $\mM$ is homotopy equivalent to a finite-dimensional semi-projective DG $\dR$\!--module and  the DG endomorphism algebra $\dEnd_{\dR}(\mM)$ is quasi-isomorphic to a finite-dimensional DG algebra.
\end{proposition}

In \cite{Or20} we claim  that any perfect DG module is homotopy equivalent to a finite-dimensional homotopically projective DG module.
But in fact we showed that it is a direct summand of a semi-free DG module, and, hence, it is semi-projective by 3) of Definition \ref{spr}.

Proposition \ref{obj} and Lemma \ref{fdmod} directly imply the following corollary.

\begin{corollary}\label{subfd}
 Let $\dR$ be a finite-dimensional DG algebra. Then the subcategory $\D\subseteq\D_{\fd}(\dR),$  generated by the DG $\dR$\!--module $\dR/\rde,$ contains the category
$\prf\dR,$ and, hence, for any DG ideal $\mI\subseteq \mJ_{+}$ the morphism
$\dR\to \dR/\mI$ is acceptable.
\end{corollary}

For a usual finite-dimensional algebra $R$ we can show more. The argument of Lemma \ref{fdmod} proves that the semisimple module $S=R/J,$ where $J$ is the radical, generates the whole category
$\D_{\fd}(R).$  Furthermore, the finite-dimensional modules generate the whole
category $\D_{\fd}(R),$ because any object of $\D_{\fd}(R)$ can be obtained from its cohomology modules by consequent application  of mapping cones.
Thus, we obtain the following proposition.

\begin{proposition} Let $R$ be a finite-dimensional algebra and let $J\subset R$ be the radical. For any  ideal $I\subseteq J$ the morphism
$R\to R/I$ is respectable.
\end{proposition}

This proposition can be generalized to the case of  non-positive finite-dimensional DG algebras.
Recall  that a DG algebra $\dR$ is called non-positive if $\dR^i=0$ for all $i>0.$ In this case, the triangulated category $\D_{\fd}(\dR)$
has a t-structure such that the forgetful functor to the derived category of $\kk$\!--vector spaces $\D(\kk)$ is exact.
Therefore, by the same argument as above for usual algebras, the DG module $\dR/\mJ_{+}$ generates the whole triangulated category $\D_{\fd}(\dR),$ and we obtain the following proposition.

\begin{proposition} Let $\dR$ be a finite-dimensional nonpositive DG algebra. For any DG ideal $\mI\subseteq \mJ_{+}$ the morphism
$\dR\to \dR/\mI$ is respectable.
\end{proposition}

\subsection{Smooth finite-dimensional DG algebras and twisted tensor products}
In this section, we consider the case of an arbitrary finite-dimensional DG algebra. The main problem here is that there are examples of  finite-dimensional DG algebras $\dR$
and a cohomologically finite-dimensional DG $\dR$\!--module $\mM$ that does not have a finite-dimensional model (see \cite{Ef}). This means that the DG module $\mM$ is not quasi-isomorphic to
any finite-dimensional DG $\dR$\!--module. In particular, we cannot conclude that the semisimple DG module $\dR/\mJ_{+}$ generates the whole category $\D_{\fd}(\dR).$

Let $\dR$ be a finite-dimensional DG algebra.
Consider, as above, the radical $\rd$ of the algebra $R$ and  denote by $\mJ_p$ the internal DG ideals $(\rd^p)_{-}$ for $p=1,\ldots, n.$
We obtain  a chain of  DG ideals  $\rdi=\mJ_{1}\supseteq\cdots\supseteq\mJ_n=0.$
Consider  DG $\dR$\!--modules $\mM_p=\dR/\mJ_p$ for $p=1,\ldots, n$ as objects of the DG category $\Mod \dR$ of all right DG $\dR$\!--modules.
Denote by
$\dE$  the DG algebra of endomorphisms $\dEnd_{\dR}(\mM)$ of the DG $\dR$\!--module $\mM=\bigoplus_{p=1}^{n} \mM_p$ in $\Mod \dR.$

Let us take the right DG $\dE$\!--module $\mP_n=\dHom_{\dR}(\mM, \mM_n).$ It is semi-projective, and  we have $\dEnd_{\dE}(\mP_n)\cong\dR.$
Thus, the DG $\dE$\!--module $\mP_n$ is actually a DG $\dR\hy\dE$\!--bimodule, and it induces two functors
\[
(-)\stackrel{\bL}{\otimes}_{\dR} \mP_n: \D(\dR)\lto\D(\dE)\quad\text{and}\quad
\bR\Hom_{\dE}(\mP_n, -):\D(\dE)\lto\D(\dR)
\]
that are adjoint to each other. The DG $\dE$\!--module $\mP_n$ is perfect and, hence, the derived functor $(-)\stackrel{\bL}{\otimes}_{\dR} \mP_n$  sends perfect modules to perfect ones.
Thus, the $\dR\hy\dE$\!--bimodule $\mP_n$ gives a functor
\[
(-)\stackrel{\bL}{\otimes}_{\dR} \mP_n: \prf\dR\lto\prf\dE.
\]

The following theorem was proved in \cite{Or20}.

\begin{theorem}\cite[Th. 2.19]{Or20}\label{DGalgebra}
Let $\dR$ be a finite-dimensional DG algebra of nilpotency index $n$ and let
$\dE=\dEnd_{\dR}(\bigoplus_{p=1}^{n} \mM_p)$ be the DG endomorphism algebra defined above. Then the following properties hold:
\begin{enumerate}
\item[1)] The functors  $\prf\dR\to\prf\dE,\;  \D(\dR)\to\D(\dE),$ given by $(-)\stackrel{\bL}{\otimes}_{\dR} \mP_n,$ are fully faithful.
\item[2)] The triangulated category $\prf \dE$ has a full semi-exceptional collection.


\item[3)] The triangulated category $\prf \dE$ is regular and there is an equivalence
$
\prf\dE\stackrel{\sim}{\to}\D_{\fd}(\dE).
$
\item[4)] If the semisimple part $\dS_{+}=\dR/\rde$ is separable, then the DG algebra $\dE$ is smooth.
\item[5)] If $\dR$ is smooth, then $\prf\dR$ is an admissible subcategory of $\prf\dE.$
\end{enumerate}
\end{theorem}

This theorem gives us a regular (or smooth) realization for any finite-dimensional DG algebra $\dR$ and allows us to prove the following proposition.

\begin{proposition}\label{newreg}
Let $\dR$ be a finite-dimensional DG algebra and let $\mI\subseteq \mJ_{+}$ be a DG ideal.
Assume that $\mf: \dR\to \dR/\mI$ is a pp-morphism and the DG algebra $\dR/\mI$ is regular. Then the morphism $\mf$ is respectable and the DG algebra $\dR$ is also regular.
If, in addition, the semisimple DG algebra $\dR/\mJ_{+}$ is separable, then $\dR$ is smooth.
\end{proposition}
\begin{proof}
Let $\D\subseteq\D_{\fd}(\dR)$ be the subcategory generated by the DG $\dR$\!--module $\dR/\mJ_{+}.$ By Corollary \ref{subfd}, the subcategory $\D$  contains the category $\prf\dR.$
On the other hand, since $\dR/\mI$ is regular, the DG  $\dR/\mI$\!--module $\dR/\mJ_{+}$ belongs to $\prf\dR/\mI.$ Hence, $\dR/\mJ_{+}$  is also perfect as a DG $\dR$\!--module, because $\mf: \dR\to \dR/\mI$ is a pp-morphism.
Thus, we obtain an equivalence $\D\cong\prf\dR.$

Let us now consider the functor
\[
\bR\Hom_{\dE}(\mP_n, -):\D_{\cf}(\dE)\lto\D_{\cf}(\dR).
\]
The category $\D_{\cf}(\dE)$ is equivalent to $\prf\dE.$ Consider the subcategory $\D'\subseteq\D_{\cf}(\dR)$ that is generated by the image of the functor
$\bR\Hom_{\dE}(\mP_n, -).$ Actually, it is generated by the DG-modules $\mM_p=\dR/\mJ_p$ for $p=1,\ldots, n,$ because $\bR\Hom_{\dE}(\mP_n, \dE)=\bigoplus_{p=1}^{n} \mM_p.$
Thus, by Lemma \ref{fdmod}, the subcategory $\D'$ is generated by $\dR/\mJ_{+}\cong \dR/\mJ_{-}=M_1$ and it coincides with the subcategory $\D\cong \prf\dR.$
This implies that the functor $\bR\Hom_{\dE}(\mP_n, -)$ sends $\prf\dE$ to $\prf\dR.$ Hence, the full embedding $\prf\dR\hookrightarrow \prf\dE$ has a right adjoint.
The DG algebra $\dE$ is regular, i.e. $\dE$ is a strong generator for $\prf\dE.$ Therefore, the object $\bR \Hom_{\dE} (\mP_n, \dE)$ is a strong generator
for $\prf\dR,$ and $\dR$ is regular too. Thus, $\prf\dR$ is equivalent to $\D_{\cf}(\dR)$ and $f$ is respectable.

Suppose that the semisimple DG algebra $\dR/\mJ_{+}$ is separable. Theorem \ref{DGalgebra} implies that the DG algebra $\dE$ is smooth.
Finally, by \cite[3.24]{LS}, the DG algebra $\dR$ is smooth too because $\prf\dR$ is an admissible subcategory of a smooth category.
\end{proof}

Since $\dR/\mJ_{+}$ is regular as a semisimple DG algebra, Propositions \ref{newreg} and \ref{BonVdB} imply the corollary.

\begin{corollary} A finite-dimensional DG algebra  $\dR$ is regular if and only if the DG $\dR$\!--module $\dR/\mJ_{+}$ is perfect.
If, in addition, the semisimple DG algebra $\dR/\mJ_{+}$ is separable, then $\dR$ is smooth.
\end{corollary}

Now we can apply these results to DG twisted tensor products of finite-dimensional DG algebras.

\begin{theorem}\label{DGfinpr}
Let $\dR$ be a finite-dimensional DG algebra. Let $\dA$ and $\dB$ be finite-dimensional $\dR$\!--rings  such that $\dA$ has an augmentation $\pi_A: \dA\to \dR$ with an ideal $\mI_{\dA}$ and $\dB$ is semi-flat as a left DG $\dR$\!--module.
Let $\dC^{\nabla}=\dA\otimes_{\dR}^{\nabla, \tau}\dB$ be a DG twisted tensor product.
 Assume that  $\dA, \dB$ are regular and the ideal $\mI_{\dA}\otimes_{\dR}\dB\subset \dC^{\nabla}$ is contained in the external radical $(\mJ_{\dC^{\nabla}})_{+}.$
Then, the DG algebra $\dC^{\nabla}=\dA\otimes_{\dR}^{\nabla, \tau}\dB$  is also regular. If, in addition, the semisimple DG algebra $\dB/(\mJ_{\dB})_{ +}$ is separable, then $\dC^{\nabla}$ is smooth.
\end{theorem}
\begin{proof} Since $\dA$ is regular and $\dR$ is proper, the DG $\dA$\!--module $\dR$ belongs to $\prf\dA.$ Thus, $\pi_A$ is a pp-morphism.
Sine $\dC^{\nabla}$ is semi-flat as the left DG $\dA$\!--module, the DG $\dC^{\nabla}$\!--module
$\dB\cong \dR\otimes_{\dA}\dC^{\nabla}$ is perfect.
This means that the morphism $p_B: \dC^{\nabla}\to \dB$ is also a pp-morphism. Applying Proposition \ref{newreg} to the morphism $p_B,$ we obtain that
the morphism $p_B: \dC^{\nabla}\to\dB$ is respectable and the DG algebra $\dC^{\nabla}$ is regular.

The preimage of the radical $J_{B}\subset\dB$ under the surjective morphism $p_B$ is the sum of the radical $J_{C}$ and $ \Ker p_B.$ Since $\Ker p_B=\mI_{\dA}\otimes_{\dR}\dB$ is contained in the external radical $(\mJ_{\dC^{\nabla}})_{+},$ the preimage of the external radical $(\mJ_{\dB})_{+}$
coincides with $(\mJ_{\dC^{\nabla}})_{+}.$
When $\dC^{\nabla}/(\mJ_{\dC^{\nabla}})_{+}\cong\dB/(\mJ_{\dB})_{+}$ is separable, the
DG algebra $\dC^{\nabla}$ is also smooth.
\end{proof}

\section{Algebras and DG algebras with two simple modules}

\label{ExampAlg}

\subsection{Quivers with two vertices and Kronecker algebras}
In this section, we consider some smooth basic algebras (and also DG algebras) with two simple modules, i.e $\kk$\!--algebras $R$ such that the semisimple part $R/J$ is isomorphic
to the algebra $\kk\times \kk.$ The main goal is, on the one hand, to demonstrate how such algebras can be obtained from simpler ones using twisted tensor product, and, on the other hand, to construct new smooth algebras.

Let $Q$ be a quiver. It consists of the data
$(Q_0, Q_1, s, t),$
where $Q_0, Q_1$ are finite sets of vertices and arrows, respectively, while
$s, t : Q_1\to  Q_0$
are maps associating to each arrow its source and target.
The path algebra $\kk Q$ is  determined
by the generators $e_q$ for $q\in Q_0$ and $a$ for $a\in Q_1$
with the following  relations:
$
e^2_q = e_q,\quad e_r e_q = 0,$
 when
 $r\ne q,$
and
$
e_{t(a)} a = a e_{s(a)} = a.
$
As a $\kk$\!--vector space,
the path algebra $\kk Q$ has a basis consisting of the set of all paths in $Q,$
where a path $\overline{p}$ is a possibly empty sequence $a_{l} a_{l-1}\cdots a_1$ of compatible arrows,
i.e. $s(a_{i+1})=t(a_i)$ for all $i=1,\dots, l-1.$
For an empty path we have to choose a vertex of the quiver.
The composition of two paths $\overline{p}_1$ and $\overline{p}_2$
is defined as $\overline{p}_2 \overline{p}_1$ if they are compatible
and as $0$ if they are not compatible.

To obtain a basic finite-dimensional algebra, we have to consider a quiver with
relations.
A relation on a quiver $Q$ is a subspace of $\kk Q$ spanned by linear
combinations of paths having a common source and a common target, and of length at
least 2.
A quiver with relations is a pair $(Q, I),$ where $Q$ is a quiver and $I$ is a two-sided ideal
of the path algebra $\kk Q$ generated by relations.
The quotient algebra $\kk Q/I$ will be called the quiver algebra of the quiver with relations
$(Q, I).$
It can be shown that any basic finite-dimensional algebra can be realized as an algebra of
 some quiver with relations $(Q, I)$ (see \cite{Ga}).

Let $Q_{n,m}$ be a quiver with two vertices $\bone, \btwo$ and with $n$ arrows $\{c_1,\ldots, c_n\}$  from $\bone$ to $\btwo$ and $m$ arrows $\{b_1,\ldots, b_m\}$
 from  $\btwo$ to $\bone.$
\begin{equation}\label{TwoVQ}
Q_{n,m}=\Bigl[
\xymatrix{
\underset{\btwo}{\bullet}\ar@/_0.5pc/[rr]_{b_1} \ar@{{ }{ }}@/_1pc/[rr]_{\vdots} \ar@/_3pc/[rr]_{b_m} & &
\underset{\bone}{\bullet}\ar@/_0.5pc/[ll]_{c_1}  \ar@{{ }{ }}@/_1.5pc/[ll]_{\vdots} \ar@/_3pc/[ll]_{c_n}
}
\Bigr].
\end{equation}

Let $\kk Q_{n,m}$ be the path algebra of the quiver $Q_{n,m}.$
Denote by $B$ and $C$ the vector spaces generated by the arrows $\{b_1,\ldots, b_m\}$ and $\{c_1,\ldots, c_n\},$ respectively.
\begin{remark} {\rm
We do not consider quivers with loops, because quiver algebras with relations for quivers with loops have infinite global dimensions
(see \cite[Cor. 5.6]{Ig}).
}
\end{remark}

The first block of natural examples of smooth algebras is obtained by considering the path algebras of the quivers $Q_{n,0}$ (or $Q_{0,m}$).
They are called Kronecker quivers, and their path algebras have global dimension $1.$ Denote by $K_n$ the Kronecker algebras $\kk Q_{n,0}$ with natural augmentations of the semisimple part.
It is an easy exercise to check the following statement.

\begin{proposition}
There is an isomorphism of algebras $K_n\cong K_1\otimes^{\mathbf{v}}_S K_{n-1}\cong K_{n-1}\otimes^{\mathbf{v}}_S K_{1},$ where $S=\kk\times\kk$ is the semisimple part and the twisting map $\mathbf{v}$
is defined by formula (\ref{vtwist}).
\end{proposition}

We can also consider a DG version of such algebras if we assume that each arrow $c_i$ has some degree $d_i\in\ZZ.$ Denote by $K_n[d_1,\dots, d_n]$ the DG algebra
of such a DG quiver. The same arguments give us that any DG algebra $K_n[d_1,\dots, d_n]$ is equal to $K_1[d_n]\otimes^{\mathbf{v}}_S K_{n-1}[d_1,\cdots, d_{n-1}].$
The triangulated category $\prf K_n[d_1,\dots, d_n]$ has a full exceptional collection consisting of two objects, and any proper category with  an exceptional collection of two
objects is equivalent to the category $\prf K_n[d_1,\dots, d_n]$ for some $n$ and $d_i$ (see \cite[Prop. 3.8]{Or16} for a general statement).

Another generalization is to consider a DG twisted tensor product with nontrivial $\nabla.$ The simplest example is the DG algebra
$K_1\otimes^{\nabla, \mathbf{v}}_S K_{1}[-1]$ with $\nabla(c_2)=c_1.$ The resulting DG algebra is quasi-isomorphic to $K_0=S.$

\subsection{Green algebras}

The second block of examples of smooth algebras that we are going to discuss appeared in E.L.~Green's paper \cite{G} and was discussed in detail by D.~Happel in \cite{Ha} (see also \cite{HZ}).
Let us take the quiver algebras of the quivers $Q_{n,n}$ or $Q_{n,n-1}$ with the following relations:
\smallskip

\begin{tabular}{lll}
(1) & $c_jb_i=0,$ & when $j\le i;$\\
(2) & $b_ic_j=0,$ & when $i<j.$
\end{tabular}
\smallskip

Denote these algebras by $G_{2n}$ and $G_{2n-1},$ respectively. It is proved in \cite{Ha} that the algebra $G_{k}$ has finite global dimension, which is equal to $k.$
For small $k,$ we have $G_0=K_0=S$ and $G_1=K_1.$ Moreover, it is not so difficult to check that the category $\prf G_2$ is equivalent to $\prf K_2[0,1],$ and, hence, it has a full exceptional collection.
On the other hand, it was proved in \cite{Ha, MH} that  the algebras $G_{k}$ are not derived equivalent to quasi-hereditary algebras for $k\ge 3.$
Therefore, the categories $\prf G_k$ do not have full exceptional collections when $k\ge 3$ (see \cite{LY}).

Let us now show that each algebra $G_k$ can be realized as an iterated twisted tensor product of  algebras of type $K_1.$
Denote by $K_1^{\op}$ the Kronecker algebra $\kk Q_{0,1}.$ This algebra is opposite to the algebra $K_1=\kk Q_{1,0},$ and it is also isomorphic to $K_1.$

\begin{proposition} For any $n\ge 1$ there are isomorphisms of algebras
\[
G_{2n+1}\cong K_1 \otimes^{\mathbf{v}}_S G_{2n} \quad\text{and}\quad G_{2n}\cong K_1^{\op}\otimes^{\mathbf{v}}_S G_{2n-1}.
\]
\end{proposition}
\begin{proof} The proof is a direct calculation. It goes by induction, where the first factors $K_1$ and $K_1^{\op}$ are related to the arrows $c_{n+1}$ and $b_n,$ respectively.
\end{proof}

There are some generalizations of these algebras. As above, we can consider a DG version taking into account that each arrow $b_i$ and $c_j$ could have a certain degree.
Another generalization is that, instead of algebras of type $K_1$ (and $K_1^{\op}$), one can take algebras of type $K_q.$ Thus, taking iterated twisted tensor products we can obtain
algebras of the form
\begin{equation}\label{gGreen}
G_{\langle p_n, q_n,\dots, p_1, q_1\rangle}=K_{p_n}^{\op}\otimes^{\mathbf{v}}_S K_{q_n}\otimes^{\mathbf{v}}_S\cdots \otimes^{\mathbf{v}}_S K_{p_1}^{\op}\otimes^{\mathbf{v}}_S K_{q_1}
\end{equation}

In the case when $p_i=q_i=1$ for all $1\le i\le n,$  we obtain a usual algebra $G_{2n}.$ In the case, when all $p_i, q_i$ are equal to $1,$ except for the last one $p_n,$ which is equal to $0,$ we get $G_{2n-1}.$

Let us also consider another generalization of Green algebras to the case of $N$ simple modules.
We fix a natural number $N\ge 2$ and consider the semisimple algebra $S=S_N=\underbrace{\kk\times\dots\times\kk}_N.$ Denote by $K_{ij}[d]$ for $1\le i\ne j\le N$  the finite-dimensional DG algebras with the semisimple part equal to $S=S_N$ and with only one arrow  from vertex
$i$ to $j$ of degree $d$ in the Jacobson radical.
Taking iterated twisted tensor product over $S=S_N$  of such DG algebras  with the twisting map given by formula (\ref{vtwist}) from Construction \ref{Mtwist}, we obtain
new DG algebras, which will be called {\sf generalized Green DG algebras on $N$ vertices}.
By Definition \ref{TwistDGal} of the twisted tensor product of DG algebras,  any such generalized Green DG algebra has a trivial differential.
The next proposition  directly follows from Theorem \ref{DGfinpr}.

\begin{proposition}\label{GreenDG}
Any iterated twisted tensor product of DG algebras $K_{ij}[d]$ over the semisimple part $S=S_N$  via twisting (\ref{vtwist}) is a smooth DG algebra.
\end{proposition}

In the Section \ref{GroGr} we discuss  bilinear forms on the Grothendieck group of these DG algebras.

\subsection{New examples of smooth algebras}\label{Newalg}
Now we introduce and describe new families of smooth algebras with two simple modules. For simplicity, we assume that the base field $\kk$ is infinite. We start with a quiver of the form $Q_{m,n},$ as in (\ref{TwoVQ}), and fix some new relations.
These relations will also depend on some integer $0<k< n.$ It is important to note that these families of algebras have nontrivial moduli spaces.

 Let $V_i\subset C, i=1,\dots,m,$ be some subspaces of dimension $k$ and $W_i\subset C, i=1,\dots,m,$ be some subspaces of codimension $k.$
 Denote by $\mathfrak{F}=\{V_1,\ldots, V_m; W_1,\ldots, W_m\}$ the resulting family of subspaces.
We consider the following relations $I_{\mathfrak{F}}$ depending on the family $\mathfrak{F}:$
\smallskip

\begin{tabular}{lll}

(1) & $bcb'=0$ & for any $c\in C, b, b'\in B,$\\

(2) & $w b_i=0$ & for any $w\in W_i,$ where $i=1,\dots, m,$\\

(3) & $b_i v=0$ & for any $v\in V_i,$ where $i=1,\dots, m.$\\

\end{tabular}

\medskip
Denote by $R_{\mathfrak{F}}$ the quotient algebra $R_{\mathfrak{F}}=\kk Q_{n,m}/I_{\mathfrak{F}}.$
Let $S_1$ and $S_2$ be the right simple $R_{\mathfrak{F}}$\!--modules, and let $P_1=e_1R_{\mathfrak{F}}$ and $P_2=e_2 R$ be the  right projective $R_{\mathfrak{F}}$\!--modules.
There are natural short
exact sequences of right modules:
\begin{equation}\label{simplshort}
0\to BR_{\mathfrak{F}}\to P_1\to S_1\to 0\quad \text{and}\quad 0\to CR_{\mathfrak{F}}\to P_2\to S_2\to 0,
\end{equation}
where $CR_{\mathfrak{F}}=\sum_{c\in C} cR_{\mathfrak{F}}$ and $BR_{\mathfrak{F}}=\sum_{b\in B} bR_{\mathfrak{F}}$ are the right ideals generated by $C$ and $B.$

Let us choose  subspaces $V\subset C$ and $W\subset C$ of dimensions $k$ and $n-k,$ respectively,  such that $V\cap W_i=0$ and $W\cap V_i=0$ for all $i=1,\dots, m.$
We also assume that $V\cap W=0$ and, hence, $V\oplus W=C.$
The relations $I_{\mathfrak{F}}$ imply the following decompositions.
\begin{lemma} \label{decompR}
There are the following decompositions of vector spaces:
\[
e_1 R_{\mathfrak{F}} e_1=\langle e_1\rangle \oplus \bigoplus_{j=1}^m b_j W ,\quad
e_1 R_{\mathfrak{F}} e_2= B,\quad
e_2 R_{\mathfrak{F}} e_1=C\oplus  \bigoplus_{j=1}^m V b_j W,\quad
e_2 R_{\mathfrak{F}} e_2=\langle e_2\rangle \oplus \bigoplus_{j=1}^m V b_j.
\]
\end{lemma}

Lemma \ref{decompR} implies that the right $R_{\mathfrak{F}}$\!--module $BR_{\mathfrak{F}}$ is isomorphic to the direct sum $ \mathop\oplus\limits_{i=1}^{m} b_i R_{\mathfrak{F}}.$
Moreover, the following statement holds.
\begin{lemma} \label{BR}
For any subset $\{j_1,\cdots, j_l\}\subseteq \{ 1, \cdots, n \}$ the right ideal $b_{j_1}R_{\mathfrak{F}}+\cdots + b_{j_l}R_{\mathfrak{F}}\subset R_{\mathfrak{F}}$
is isomorphic to the direct sum $b_{j_1} R_{\mathfrak{F}}\oplus\cdots\oplus b_{j_l} R_{\mathfrak{F}}$ as a right $R_{\mathfrak{F}}$\!--module.
\end{lemma}

For any element $r\in R_{\mathfrak{F}}$ denote by $\rAnn(r)\subset R_{\mathfrak{F}}$ the right annihilator of $r.$
The following lemma directly follows from Lemma \ref{decompR}.

\begin{lemma}\label{xb} For any $x\in C\backslash W_i,$ there are isomorphisms of right modules $xb_i R_{\mathfrak{F}}\cong b_i R_{\mathfrak{F}}$ and
$\rAnn (xb_i)\cong\rAnn (b_i).$
\end{lemma}

Let $U_{ij}\subseteq C$ denote the sum of the subspaces $V_i$ and $W_j.$ We put
\[
t_{ij}=\codim U_{ij}=\dim (V_i\cap W_j).
\]
Denote by $t$ the maximum of all $t_{ij}.$
Let us choose a sequence of elements $x_1,\cdots, x_t\in C$ such that $x_1$ does not belong to any proper subspace $U_{ij}$ and  $x_{q+1}$ does not belong to any subspace $U_{ij}+\langle x_1,\cdots, x_q\rangle,$ when $t_{ij}> q.$


For any element $c\in C$ we denote by $T_c\subseteq \{1,\ldots, m\}$ the subset consisting of all $l$ such that $c\in W_l.$
Let us compute the right annihilators of certain elements.

\begin{lemma}\label{L1} There are the following isomorphisms of the right modules:
\[
\begin{array}{ll}
1) & \rAnn (c) \cong e_2 R_{\mathfrak{F}} \oplus  \mathop{\bigoplus}\limits_{j\in T_c} b_j R_{\mathfrak{F}}, \\
2) & \rAnn (b_i) \cong e_1 R_{\mathfrak{F}} \oplus V_i R_{\mathfrak{F}} \oplus \left(\mathop{\bigoplus}\limits_{j=1}^{m} \mathop{\bigoplus}\limits_{s=1}^{t_{ij}} x_sb_j R_{\mathfrak{F}} \right)
\cong e_1 R_{\mathfrak{F}} \oplus V_i R_{\mathfrak{F}} \oplus \mathop{\bigoplus}\limits_{j=1}^{m} (b_j R_{\mathfrak{F}} )^{\oplus t_{ij}}.
\end{array}
\]
\end{lemma}
\begin{proof} Equality 1) directly follows from relation (2) of  $I_{\mathfrak{F}}$ and Lemma \ref{BR}.

By Lemma \ref{decompR}, the vector space $CR_{\mathfrak{F}}$ can be decomposed as $C\oplus \bigoplus_{j=1}^m V b_j\oplus \bigoplus_{j=1}^m V b_j W,$  where
$V\subset C$ and $W\subset C$ are of dimension $k$ and $n-k,$ respectively,  such that $V\cap W_i=0$ and $W\cap V_i=0$ for all $i=1,\dots, m.$
Now, the annihilator
$\rAnn (b_i)$ is equal to $e_1 R_{\mathfrak{F}} \oplus V_i \oplus \bigoplus_{j=1}^m V b_j\oplus \bigoplus_{j=1}^m V b_j W.$ Since $\dim (V_i\cap W_j)=t_{ij},$ the codimension
of the subspace $V_ib_j$ in the space $V b_j$ is equal to $t_{ij}.$ By the construction above, the vectors $x_sb_j,$ where $s=1,\dots, t_{ij},$ give a basis for a complementary subspace in $V b_j.$
Taking into account Lemmas \ref{BR} and \ref{xb}, we obtain equalities 2).
\end{proof}

Using the previous lemma, we get the following exact sequences:
\begin{equation}\label{Shortres}
0\to \bigoplus_{j\in T_{c}} b_j R_{\mathfrak{F}}\to P_1\to c R_{\mathfrak{F}}\to 0, \qquad 0\to V_i R_{\mathfrak{F}} \oplus \mathop{\bigoplus}\limits_{j=1}^{m} (b_j R_{\mathfrak{F}})^{\oplus t_{ij}}\to P_2\to b_iR_{\mathfrak{F}}\to 0.
\end{equation}
The first short sequence implies the following proposition.
\begin{proposition}\label{pdcR} Let $c$ be an element of $C.$ If $T_{c}=\emptyset,$ then $c R_{\mathfrak{F}}\cong P_1$ and
 $\pd(c R_{\mathfrak{F}})=0.$ If $T_{c}\ne \emptyset,$ then  $\pd(cR_{\mathfrak{F}})=\max\{\pd(b_j R_{\mathfrak{F}})|\; j\in T_{c}\}+1.$
 \end{proposition}

Furthermore, for any subspace $U\subseteq C$ of dimension $d$ denote by $t_{U,j}$ the dimension of the intersection
$U\cap W_j.$ The right module $UR_{\mathfrak{F}}$ can be covered by the projective module $P_1^{d}$  and it is a minimal projective cover.
It is easy to see that for any subspace $U\subseteq C,$ there is the following short exact sequence:
\begin{equation}\label{Ures}
0\lto \mathop{\bigoplus}\limits_{j=1}^{m} (b_j R_{\mathfrak{F}})^{\oplus t_{U,j}}\lto P_1^{d}\lto U R_{\mathfrak{F}}\lto 0.
\end{equation}
In particular, for the right modules $V_i R_{\mathfrak{F}}$ and the module $CR_{\mathfrak{F}}$ we have  short exact sequences:
\begin{equation}\label{Cres}
0\to \mathop{\bigoplus}\limits_{j=1}^{m} (b_j R_{\mathfrak{F}})^{\oplus t_{i,j}}\to P_1^{k}\to V_i R_{\mathfrak{F}}\to 0
\quad\text{and}\quad
0\to \bigoplus_{i=1}^{m} (b_i R_{\mathfrak{F}})^{\oplus (n-k)}\to P_1^{n }\to CR_{\mathfrak{F}}\to 0.
\end{equation}

The following generalization of Proposition \ref{pdcR} holds.
\begin{proposition}\label{prdbR}
Let $U\subseteq C$ be a subspace of dimension $d.$ If $t_{U,j}\ne 0$ for some $j,$ then $\pd(UR_{\mathfrak{F}})=\max\{\pd(b_j R_{\mathfrak{F}})|\; t_{U,j}\ne 0 \}+1;$
otherwise, $UR_{\mathfrak{F}}\cong P_1^d.$ For any $i,$ there is an equality $\pd(b_iR_{\mathfrak{F}})=\pd(V_iR_{\mathfrak{F}})\}+1.$
\end{proposition}
\begin{proof} The first statement follows from the short exact sequence (\ref{Ures}). When $t_{U, j}=0$ for all $j,$ we have an isomorphism $UR_{\mathfrak{F}}\cong P_1^d.$
If some $t_{U, j}\ne 0,$ then $\pd(UR_{\mathfrak{F}})=\pd(\oplus_{j=1}^{m} (b_j R_{\mathfrak{F}})^{\oplus t_{U,j}})+1.$ Therefore, we obtain that $\pd(UR_{\mathfrak{F}})=\max\{\pd(b_j R_{\mathfrak{F}})|\; t_{U,j}\ne 0 \}+1.$

The second short sequence of (\ref{Shortres}) implies that $\pd(b_iR_{\mathfrak{F}})=\pd(V_iR_{\mathfrak{F}})+1$ because $\pd(V_i R_{\mathfrak{F}})=\max\{\pd(b_j R_{\mathfrak{F}})|\; t_{ij}\ne 0 \}+1$ is bigger than
$\pd(b_j R_{\mathfrak{F}})$ for every $j$ with $t_{ij}\ne 0.$
\end{proof}

To any algebra $R_{\mathfrak{F}} $ we attach a quiver $\Gamma_{\mathfrak{F}},$ which encodes intersections of the subspaces in the family $\mathfrak{F}$ and allows us to construct  minimal resolutions for modules.
Any such quiver $\Gamma_{\mathfrak{F}}$ has $2m$ vertices $b_1,\cdots, b_m, v_1,\cdots, v_m.$ Numbers of arrows between vertices in this quiver depends on the numbers $t_{ij}=\dim(V_i\cap W_j).$
Namely,  arrows in $\Gamma_{\mathfrak{F}}$ are defined  by the following rule:
\begin{itemize}
\item[(1)] one arrow $b_i\to v_i$ for each $i=1,\dots,m;$
\item[(2)] $t_{ij}$ arrows from $b_i$ to $b_j$ for each pair $i,j;$
\item[(3)] $t_{ij}$ arrows from $v_i$ to $b_j$ for each pair $i,j.$
\end{itemize}

Using the quiver $\Gamma_{\mathfrak{F}},$ we can describe  minimal projective resolutions for the modules $b_iR_{\mathfrak{F}}.$ Denote by $\P_{\Gamma_{\mathfrak{F}}}$ the set of all paths in $\Gamma_{\mathfrak{F}}.$
Let $\P_{b_i}\subset \P_{\Gamma_{\mathfrak{F}}}$ be the subset of all paths $p$ the source of which, $s(p),$ coincides with $b_i.$ For any $i,$ the set $\P_{b_i}$ is divided into  a disjoint union $\bigsqcup_{s\ge0} \P^s_{b_i},$ where $\P_{b_i}^s$ are the sets of all paths of  length $s.$

Iterating the short exact sequences from (\ref{Shortres}) and (\ref{Cres}) and using the combinatorics of the quiver $\Gamma_{\mathfrak{F}},$ we can produce  minimal resolutions of the right modules  $b_i R_{\mathfrak{F}}$
for any $i.$ They have  the following  form:

\[
\cdots\to \bigoplus_{\substack{p\in \P_{b_i}^s \\ t(p)=v_j}} P_1^k \oplus \bigoplus_{\substack{p\in \P_{b_i}^s\\  t(p)=b_j }} P_2\to\cdots \to P_1^k \oplus \bigoplus\limits_{\substack{p\in
\P_{b_i}^1\\ t(p)=b_j}} P_2 \to P_2\to b_iR_{\mathfrak{F}}\to 0.
\]

These resolutions and Proposition  \ref{prdbR} imply the following theorem.

\begin{theorem}\label{fgdim} Let $Q_{n,m}$ and $\mathfrak{F}$ be as above.
 The algebra $R_{\mathfrak{F}}=\kk Q_{n,m}/I_{\mathfrak{F}}$ has finite global dimension if and only if the quiver $\Gamma_{\mathfrak{F}}$ does not have oriented cycles.
Moreover, the global dimension of the algebra $R_{\mathfrak{F}}$ is equal to $l+2,$ where $l$ is the maximal length of a path in $\Gamma_{\mathfrak{F}}.$
\end{theorem}
\begin{proof}
By Proposition  \ref{prdbR}, we have $\pd(b_iR_{\mathfrak{F}})=\pd(V_iR_{\mathfrak{F}})\}+1=\max\{\pd(b_j R_{\mathfrak{F}})|\; t_{i,j}\ne 0 \}+2.$
If there is a cycle in $\Gamma_{\mathfrak{F}},$ then $\pd(b_iR_{\mathfrak{F}})=\pd(b_iR_{\mathfrak{F}})+s$ for some $i$ and $s>0.$ This implies that the projective dimension of the  module
$\pd(b_iR_{\mathfrak{F}})$ is infinite.

On the other hand, if the quiver $\Gamma_{\mathfrak{F}}$ does not have oriented cycles, the projective dimensions of the modules $b_iR_{\mathfrak{F}}$ is equal to the length of the longest path $p$ in $\Gamma_{\mathfrak{F}}$
with the source $s(p)$ equal to $b_i.$ Finally,  short exact sequences (\ref{simplshort}) give us the following equalities for the simple modules:
\[
\pd(S_1)=\max\{\pd(b_j R_{\mathfrak{F}})|\; j \}+1
\quad\text{and}
\quad
\pd(S_2)=\pd(CR_{\mathfrak{F}})+ 1=\max\{\pd(b_j R_{\mathfrak{F}})|\; j \}+2.
\]
Therefore, the global dimension of $R_{\mathfrak{F}}$ is equal to $\pd(S_2)=\max\{\pd(b_j R_{\mathfrak{F}})|\; j \}+2=l+2,$ where $l$ is the maximal length a path in $\Gamma_{\mathfrak{F}}.$
\end{proof}

It is easy to see that the maximal length of a path in $\Gamma_{\mathfrak{F}}$ is an odd number and any odd positive integer up to $2m-1$ can be realized. Thus, we obtain the following corollary.
\begin{corollary}
For any $1\le s\le m,$ there is a family $\mathfrak{F}=\{V_1,\ldots, V_m; W_1,\ldots, W_m\}$ of subspaces of the space $C$ such that the algebra $R_{\mathfrak{F}}=\kk Q_{n,m}/I_{\mathfrak{F}}$ has global dimension equal to $2s+1.$
\end{corollary}

For subspaces $V_i$ and $W_j$ in general position, we have $V_i\cap W_j=0$ for all $1\le i, j\le m.$ In this case, the maximal length of  paths in $\Gamma_{\mathfrak{F}}$ is equal to $1,$ and we obtain the following corollary.

\begin{corollary}
Suppose $V_i\cap W_j=0$ for any $i,j.$ Then $\pd(V_i R_{\mathfrak{F}})=0$ and $\pd(b_i R_{\mathfrak{F}})=1$ for each $i,$ and the minimal projective resolutions of the simple modules have the following forms
\[
0\to P_1^{mk}\to P_2^{m}\to P_1\to S_1\to 0,\qquad
0\to P_1^{mk(n-k)}\to P_2^{m(n-k)}\to P_1^{ n}\to P_2\to S_2\to 0.
\]
In this case, the global dimension of the algebra $R_{\mathfrak{F}}$ is equal to $3.$
\end{corollary}

\begin{example} {\rm
 Let $n=m\ge 2$ and $k=1,$ i.e. $\dim V_i=1.$ Let us  choose some elements $a_i\in C,$ for which $V_i=\langle a_i\rangle.$
Define $W_i\subset C$ as the subspaces of codimension $1$ generated by the elements $\{ a_n, \dots, a_{i+1}, a_i-a_1,\dots, a_2-a_1\}.$
 The algebras $R_{\mathfrak{F}}$ obtained in this way are exactly those constructed in \cite{KK}.
 The global dimension of these algebras is equal to $2m+1.$
 }
\end{example}

We now describe the algebras $R_{\mathfrak{F}}$ as iterated twisted tensor products of certain already known algebras.
We assume that $R_{\mathfrak{F}}$ has finite global dimension, i.e. it is smooth.

First, consider the case $m=1,$ i.e. the case where the family $\mathfrak{F}$ consists of two subspaces $V_1$ and $W_1.$
The algebra $R_{\mathfrak{F}}$ has finite global dimension if and only if $V_1\cap W_1=0.$  Denote by $K(V_1)$ and $K(W_1)$ the Kronecker subalgebras of $R_{\mathfrak{F}},$
generated by the subspaces $V_1$ and $W_1,$ respectively.
A simple direct calculation gives us that any smooth algebra $R_{\mathfrak{F}}$ for $m=1$ can be represented as a twisted tensor product of the form
$K(V_1)\otimes^{\mathbf{v}}_S K_{1}^{\op}\otimes^{\mathbf{v}}_S K(W_1).$ Thus, in this case, the algebra $R_{\mathfrak{F}}$ is a generalized Green algebra of the form $G_{\langle 0, k, 1, (n-k) \rangle}$
as in (\ref{gGreen}).

Now, consider the general case. Since $\Gamma_{\mathfrak{F}}$ does not have oriented cycles, it is a directed quiver. Hence, we may suppose, changing the numbering if necessary, that $V_m\cap W_i=0$ for all $i=1, \dots, m.$
Consider the family $\mathfrak{G},$ which consists of the subspaces $V_i, W_i$ with $i< m,$  and take the algebra $R_{\mathfrak{G}}.$
There is a natural embedding of the Kronecker algebra $K(V_m)$ into  $R_{\mathfrak{G}}$ that is given by the embedding  $V_m\subset C.$ Thus,  the algebra $R_{\mathfrak{G}}$ can be considered as a $K(V_m)$\!--ring.

\begin{proposition}\label{RleftK}
The algebra $R_{\mathfrak{G}}$ is  a projective left $K(V_m)$\!--module.
\end{proposition}
\begin{proof}
The indecomposable left projective $K(V_m)$\!--modules are $K(V_m)e_1=V_m\oplus\langle e_1 \rangle$ and $K(V_m)e_2=\langle e_2\rangle.$
By Lemma \ref{decompR}, we can represent $R_{\mathfrak{G}}$ as the direct sum of the following spaces
$(\bigoplus_{j=1}^{m-1} V_m b_j W\oplus \bigoplus_{j=1}^{m-1} b_j W ),$
$ (\bigoplus_{j}^{m-1} V_m b_j\oplus \bigoplus_{j}^{m-1} b_j),$
$(V_m\oplus \langle e_1\rangle),$  and $\langle e_2\rangle \oplus W$ for some subspace $W\subset C$ of dimension $n-k,$ for which $W\cap V_i=0$ for all $i.$
Now, the first three spaces are projective left $K(V_m)$\!--modules of type $(K(V_m)e_1)^r,$ while the last one is
isomorphic to the projective module $(K(V_m)e_2)^{(n-k+1)}.$
\end{proof}

The subspace $W_m\subset C$ determines an augmentation $\pi: R_{\mathfrak{G}}\to K(V_m)$  by setting $\pi(W_m)=0.$
We also consider the algebra $K(V_m)\otimes^{\mathbf{v}}_S K_{1}^{\op}.$ It is a $K(V_m)$\!--ring and has a natural augmentation $K(V_m)\otimes^{\mathbf{v}}_S K_{1}^{\op}\to K(V_m).$
\begin{theorem}\label{TwProd} Suppose the quiver $\Gamma_{\mathfrak{F}}$ does not have oriented cycles.
Then the algebra $R_{\mathfrak{F}}$ is isomorphic to the twisted tensor product of the algebras $(K(V_m)\otimes^{\mathbf{v}}_S K_{1}^{\op})$ and $R_{\mathfrak{G}}$ over $K(V_m)$
via the twisting map $\mathbf{v}$ defined by formula (\ref{vtwist}), i.e. we have
$
R_{\mathfrak{F}}\cong (K(V_m)\otimes^{\mathbf{v}}_S K_{1}^{\op})\otimes_{K(V_m)}^{\mathbf{v}} R_{\mathfrak{G}}.
$
\end{theorem}
\begin{proof}
For brevity, Denote the algebra $(K(V_m)\otimes^{\mathbf{v}}_S K_{1}^{\op})$ by $K(V_m; 1).$
We have an isomorphism $ K(V_m; 1)\cong K(V_m)\oplus (V_m b_m\oplus \langle b_m\rangle)$ of right $K(V_m)$\!--modules, and $(V_m b_m\oplus \langle b_m\rangle)$ is isomorphic to $S_2^{k+1},$
where $S_2$ is the simple right $K(V_m)$\!--module associated with $e_2.$

Consider the tensor product $ K(V_m; 1)\otimes_{K(V_m)}^{\mathbf{v}} R_{\mathfrak{G}}$ as a vector space.
It is isomorphic to the direct sum of $R_{\mathfrak{G}}$ and the space $(V_m b_m\oplus \langle b_m\rangle)\otimes_{K(V_m)} R_{\mathfrak{G}}.$
Using the decomposition of $R_{\mathfrak{G}}$ as a projective left $K(V_m)$\!--module, described in the proof of Proposition \ref{RleftK}, we obtain that
\[
(V_m b_m\oplus \langle b_m\rangle)\otimes_{K(V_m)} R_{\mathfrak{G}}\cong (V_m b_m\oplus \langle b_m\rangle)\otimes_{\kk} (\langle e_2 \rangle\oplus W)=(V_m b_m\oplus \langle b_m\rangle)\oplus (V_m b_m W\oplus  b_m W),
\]
because the tensor product $S_2\otimes_{K(V_m)} K(V_m)e_1$ is zero.

Applying the decomposition from Lemma \ref{decompR}, we can establish a natural isomorphism from $R_{\mathfrak{F}}$ to  the tensor product $K(V_m; 1)\otimes_{K(V_m)}^{\mathbf{v}} R_{\mathfrak{G}}$ that is an isomorphism of  vector spaces.
Thus, we have to show that the relations $I_{\mathfrak{F}}$  also hold for the twisted tensor product $K(V_m; 1)\otimes_{K(V_m)}^{\mathbf{v}} R_{\mathfrak{G}}.$
Since the isomorphism is compartible with the embeddings of the algebras $R_{\mathfrak{G}}$ and $K(V_m; 1),$
it is sufficient to check only those relations $I_{\mathfrak{F}}$ in which $b_m$ is involved.
We already know that $b_m V_m=0$ just like in the algebra $K(V_m; 1).$
The relations $b_i C b_m=0$ for $i<m$ and  $W_m b_m=0$  hold, because $W_m$ and $b_i C$ belong to the kernel of the augmentation $ \pi: R_{\mathfrak{G}}\to K(V_m).$ Now we should check the relations
$b_m C b_i=0$ for any $i\le m.$ We know that  $b_m V_m=0$ and  $W_i b_i=0$ for all $i.$ By assumption, we have $C=V_m\oplus W_i$ for any $i\le m.$ Hence, for any $i\le m,$
$b_m C b_i=0.$
\end{proof}

\begin{corollary} Any algebra $R_{\mathfrak{F}}$ can be obtained from algebras of type $K_1$ by successively applying  twisted tensor product operations.
\end{corollary}

\begin{remark}
{\rm
Note that for $m=1,$ the algebra
$R_{\mathfrak{F}}=K(V_1)\otimes^{\mathbf{v}}_S K_{1}^{\op}\otimes^{\mathbf{v}}_S K(W_1)$  can also be represented as a twisted tensor product
$(K(V_1)\otimes^{\mathbf{v}}_S K_{1}^{\op})\otimes_{K(V_1)}^{\mathbf{v}} K(C),$ where the augmentation $K(C)\to K(V_1)$ is given by specifying $W_1\subset C.$
}
\end{remark}

\begin{remark}
{\rm
Theorems \ref{TwProd} and \ref{DGfinpr} provide another proof that the algebra $R_{\mathfrak{F}}$ has finite global dimension when the quiver $\Gamma_{\mathfrak{F}}$ does not have cycles.
}
\end{remark}

\section{Smooth algebras and Grothendieck group}\label{GroGr}

\subsection{Grothendieck group and generalized Green DG algebras}

Let $\dR=(R, \dr)$ be a finite-dimensional DG algebra. Consider the Grothendieck group $K_0(\prf\dR).$
When $\dR$ is smooth, the abelian group $K_0(\prf\dR)$ is free of finite rank and, hence, it is isomorphic to $\ZZ^N$ for some $N\ge 0$ \cite[Cor. 2.21]{Or20}.
For any perfect DG modules $\mE, \mF,$ we denote by $\chi_{\dR}(\mE, \mF)$ the alternating sum
\begin{equation}\label{biform}
\chi_{\dR}(\mE, \mF)=\sum_l (-1)^l \dim_{\kk}\Hom_{\prf\dR}(\mE, \mF[l]),
\end{equation}
which defines the so-called Euler bilinear form on the abelian group $K_0(\prf\dR).$

Let us consider the case where
the semisimple part $S$ of the underlying algebra $R$ is isomorphic to $S_N\cong{\kk\times\cdots\times \kk}.$
Denote by $\{e_1,\ldots, e_N\}$ the complete set of primitive idempotents of the algebra $R.$
Moreover,  assume also that the DG algebra $\dR$ is  $S$\!--split.
This means that there are morphisms of DG algebras  $\epsilon: S\to \dR$ and  $\pi: \dR\to S$ such that $\pi\circ \epsilon=\id_S.$
In particular, $\dr(e_i)=0$ for any $i=1,\ldots, N$ and $d(\rd)\subseteq\rd,$ where $J$ is the radical.

If $\dR$ is smooth, then $K_0(\prf\dR)\cong\ZZ^N,$ and the classes of the simple modules $\mS_i$ form a basis in $K_0(\prf\dR).$ Denote by $\mP_{i}=e_i\dR, i=1,\dots, N$ the semi-projective DG modules.
Their classes also form a basis of $K_0(\prf\dR),$ and we have $\chi_{\dR} (\mP_i, \mS_j)=\delta_{ij}.$ Let $\Chi_{\dR}$ be
the matrix of the bilinear form $\chi_{\dR}$ in the basis $\mP_{i}.$ A direct calculation gives us the following equalities:
\begin{multline}\label{matrel}
(\Chi_{\dR})_{ij}=\chi_{\dR}(\mP_i, \mP_j)=\sum_l (-1)^l \dim_{\kk}\Hom_{\prf\dR}(\mP_i, \mP_j[l])
=\sum_l (-1)^l \dim_{\kk} H^l(e_j\dR e_i)=\\
=\sum_l (-1)^l \dim_{\kk} e_j\dR^l e_i.
\end{multline}
where $\dR^l$ is the degree $l$ component of $\dR.$

\begin{proposition}\label{chiprod}
Let $\dA, \dB$ be  two smooth  finite-dimensional $S$\!--split DG algebras as above, where $S\cong{\kk\times\cdots\times \kk}.$
Let $\dC^{\nabla}=\dA\otimes_S^{\nabla, \tau}\dB$ be a DG twisted tensor product over $S.$ Then, for any $\tau$ and $\nabla,$ the DG algebra $\dC^{\nabla}$ is also $S$\!--split and smooth,   and
 $\Chi_{\dC^{\nabla}}=\Chi_{\dB}\cdot\Chi_{\dA}.$
\end{proposition}
\begin{proof}
As $\dA$ is $S$\!--split, we have that $\mJ_{\dA -}=\mJ_{\dA+}=\rd_{A}.$ By Remark \ref{twistedIdeal}, the twisting map $\tau$ sends $B\otimes_S \rd_{A}$ to $\rd_{A}\otimes_S B.$
This means that the two-sided ideal $\rd_{A}\otimes_S B\subset \dC^{\nabla}$ is nilpotent and, hence, it is contained in the external radical $(\mJ_{\dC^{\nabla}})_{+}.$
Thus, by Theorem \ref{DGfinpr}, the DG algebra $\dC^{\nabla}$ is also smooth. Moreover, the morphisms of DG algebras $i_A: \dA\to \dC^{\nabla}$ and $p_B: \dC^{\nabla}\to \dB$ induce morphisms
$\epsilon_C: S\to \dC^{\nabla}$ and  $\pi_C: \dC^{\nabla}\to S$ such that $\pi_C\circ \epsilon_C =\id_S.$ Therefore, the DG algebra $\dC^{\nabla}$ is $S$\!--split too.
Let us apply equality (\ref{matrel}) to the DG algebra $\dC^{\nabla}:$
\begin{multline*}
(\Chi_{\dC^{\nabla}})_{ij}=
\sum_l (-1)^l \dim_{\kk} e_j(\dC^{\nabla})^l e_i=\sum_l (-1)^l \dim_{\kk}\left(\bigoplus_{s+t=l} \bigoplus_{r=1}^{N} e_j\dA^s e_r\otimes_{\kk} e_r \dB^t e_i\right)=\\
=\sum_{s,t} (-1)^{s+t} \dim_{\kk}\left(\bigoplus_{r=1}^{N} e_j\dA^s e_r\otimes_{\kk} e_r \dB^t e_i\right)=\sum_{s,t} (-1)^{s+t}\sum_{r=1}^N \dim_{\kk}(e_j\dA^s e_r)\cdot\dim_{\kk} (e_r \dB^t e_i).
\end{multline*}
Changing the summation order leads to the following equations:

\[
(\Chi_{\dC^{\nabla}})_{ij}=\sum_{r=1}^N \left(\sum_{s}(-1)^s \dim_{\kk} e_j\dA^s e_r\right)\cdot \left(\sum_{t}(-1)^t \dim_{\kk} e_r\dB^t e_i\right)=
\sum_{r=1}^N (\Chi_{\dB})_{ir}\cdot (\Chi_{\dA})_{rj}.
\]

Thus, we obtain the required matrix equality $\Chi_{\dC^{\nabla}}=\Chi_{\dB}\cdot\Chi_{\dA}.$
\end{proof}

In the previous section (see Proposition \ref{GreenDG}), we introduced generalized Green DG algebras. Recall that we considered  DG algebras $K_{ij}[d], 1\le i\ne j\le N$  that have only one arrow  from
$i$ to $j$ of degree $d$ in the Jacobson radical, i.e. $J=e_jJe_i\cong\kk.$ Generalized Green DG algebras were defined as iterated twisted tensor products over $S=S_N$  of such DG algebras  with the twisting map given by formula (\ref{vtwist}) from Construction \ref{Mtwist}.

The matrix $\Chi_{K_{ij}[d]}$ is equal to ${\mathtt E}_{ij}^{\epsilon},$ where ${\mathtt E}_{ij}$ is the elementary matrix and $\epsilon=(-1)^d.$ Since the elementary matrices generate the group
$\SL(n,\ZZ),$ we obtain the following corollary.

\begin{corollary}\label{anymatr} For any matrix $\Chi\in\SL(n,\ZZ),$  there is a smooth DG algebra $\dR$ such that
$\Chi_{\dR}=\Chi.$ Moreover, for any $\Chi\in\SL(n,\ZZ),$  there is a generalized Green DG algebra $\dR$ with
$\Chi_{\dR}=\Chi.$
\end{corollary}

\subsection{Grothendieck groups of algebras $R_{\mathfrak{F}}$}
Let us now consider the algebras $R_{\mathfrak{F}}$ constructed in  Section \ref{Newalg}. We fix a family ${\mathfrak{F}}$ such that the algebra
$R_{\mathfrak{F}}$ is smooth.
The Grothendieck group $K_0(\prf R_{\mathfrak{F}})$ is isomorphic to $\ZZ^2$, and the classes of the projective modules $P_1, P_2$ form a basis. Denote by $\Chi_{R_{\mathfrak{F}}}$
the matrix of the bilinear form  $\chi_{R_{\mathfrak{F}}}$ in this basis.
Since the algebra $R_{\mathfrak{F}}$ is smooth, the classes of the simple modules $S_1, S_2$ also form a basis of the Grothendieck group, and the matrix of $\chi_{R_{\mathfrak{F}}}$ in this basis is
equal to $(\Chi_{R_{\mathfrak{F}}}^{-1})^t.$ A direct calculation gives us the following matrices:
\[
\Chi_{R_{\mathfrak{F}}}=
\begin{pmatrix}
m(n-k) +1 & n+ mk(n-k)\\
m & mk+1
\end{pmatrix}
\qquad
(\Chi_{R_{\mathfrak{F}}}^{-1})^t=
\begin{pmatrix}
mk +1 & -m\\
-n- mk(n-k) & m(n-k)+1
\end{pmatrix}.
\]

Let us denote by $q_{{\mathfrak{F}}}(\mE)$ the Euler quadratic form $\chi_{R_{\mathfrak{F}}}(\mE, \mE).$ It is an integral binary quadratic form that is equal to
\[
q_{{\mathfrak{F}}}(x, y)=(m(n-k)+1) x^2 + (mk(n-k)+m+n)xy+ (mk+1)y^2
\]
in the basis of the classes of the projective modules.

\begin{proposition}
Let $q_{{\mathfrak{F}}}$ be the Euler quadratic form for the  algebra $R_{\mathfrak{F}}.$ Then
\begin{itemize}
\item[1)]
The discriminant $D(q_{{\mathfrak{F}}})$ is equal to $F^2-4,$ where $F=mk(n-k)+n-m.$
\item[2)]
If $n=2, k=1$ the form $q_{{\mathfrak{F}}}$ is  positive semidefinite for any $m\ge 1,$ and it is equivalent to the form
$q'_{{\mathfrak{F}}}=(m+1)x^2.$

\item[3)]
If $n\ge 3,$ the form $q_{\mathfrak{F}}$ is indefinite for any $n> k\ge 1$ and $m\ge 1.$
In this case, it is equivalent to the form
$
q'_{{\mathfrak{F}}}=(mk+1)x^2 + (F-2k)xy-(k(n-k)-1)y^2,
$
where $F=mk(n-k)+n-m.$

\item[4)]
The quadratic form $q_{{\mathfrak{F}}}$ does not represent $1$ for any $m\ge 1, n>k\ge 1.$
\end{itemize}

\end{proposition}
\begin{proof}
Statements from parts 1), 2), 3) are direct calculations.

Let us discuss part 4). The case $n=2$ is evident, so we assume that $n\ge 3.$ For brevity we will denote the quadratic form $ax^2+bxy+cy^2$ by $(a,b,c).$
An indefinite form $(a, b, c)$ of
discriminant $D > 0$ is called reduced if
$0 < b <\sqrt{D}$ and
$\sqrt{D} - b < 2 | a |<\sqrt{D} + b.$
It is well known that any indefinite form is equivalent to a reduced form
of the same discriminant and the number of reduced forms of a given discriminant
is finite (see, e.g., \cite[Prop. 3.3]{Bu}).
Two reduced forms $(a, b, c)$ and $(a', b', c')$ are called adjacent
if $a'=c$ and $b + b' \equiv 0 (\mod\; 2c).$ It can be shown that there is a unique reduced
form adjacent to the right of any given reduced form.
Thus, the set of reduced forms of a given discriminant
can be partitioned into cycles of adjacent forms (see, \cite[Prop. 3.4]{Bu}).
Finally, two reduced forms are equivalent if and only if they
are in the same cycle (see, \cite[Th. 3.5]{Bu}).
In our case, the principal form, i.e. a reduced form of type $(1,b,c),$ with discriminant $D=F^2-4$ looks like $(1, F-2, 2-F).$
The principal cycle of adjacent forms consists of two forms $(1, F-2, 2-F)$ and $ (2-F, F-2, 1).$
Consider the quadratic form $q'_{{\mathfrak{F}}}=(mk+1, F-2k, 1-k(n-k))$ from 3). It is reduced and it is equivalent to the  form $q_{{\mathfrak{F}}},$
but it does not belong to the principal cycle. Therefore, our form $q_{{\mathfrak{F}}}$ does not represent $1$ for any $m\ge 1, n>k\ge 1.$
\end{proof}

\begin{corollary}\label{noexc}
For any $ m\ge 1$ and $n>k\ge 1,$ the triangulated category $\prf R_{\mathfrak{F}}$ does not have exceptional objects.
\end{corollary}

\subsection{Final remarks}

Consider the DG algebra $\dR=\kk[\varepsilon]/\varepsilon^2$ with $\deg \varepsilon=\delta.$ It can be shown that $\dR$ has many different smooth realization with  Grothendieck groups of  rank $2.$
Suppose $F:\prf \dR\hookrightarrow\prf\dE$ is a full embedding, where $\dE$ is a smooth proper DG algebra, and  the right adjoint $G:\prf\dE\to \D_{\fd}(\dR)$
is a Verdier quotient. The objects $F(\dR)$ and  $K\in \prf\dE$ such that $G(K)\cong\kk$ give two elements $r, k\in K_0(\prf\dE)$ with $\chi_{\dE}(r,r)=1+(-1)^{\delta}$ and $ \chi_{\dE}(r,k)=1,$ which generate $\ZZ^2\subseteq K_0(\prf\dE).$ Hence, the rank of $K_0(\prf\dE)$ is at least $2.$
For any $p, q\ge 0,$ the DG algebras
\[
\dE_{[p,q;\delta]}=K_{p}^{\op}\otimes^{\mathbf{v}}_S K_{1}\otimes^{\mathbf{v}}_S  K_{1}^{\op}[\delta]\otimes^{\mathbf{v}}_S K_{q}
\]
provide different examples of  smooth categories $\T=\prf \dE_{[p,q; \delta]}$ with $K_0(\T)\cong\ZZ^2,$ for which there is a full embedding
$\prf \dR\hookrightarrow\T.$ It works because
$\dEnd_{\dE_{[p,q;\delta]}}(P_2)\cong \dR.$

Let us consider two of them: $\dE_{[0,0; \delta]}$ and $\dE_{[0,1; \delta]}.$ Informally, we feel that the first DG algebra is simpler than the second one. But are there invariants allowing this to be strictly determined?
Following \cite{Ro}, we can define the {\sf dimension} of a triangulated category $\T$ as the minimal integer $d\ge 0,$ for which there is a strong generator $E\in \T$
with $\langle E\rangle_{d+1}=\T.$ It can be shown that $\dim \D_{\fd}(\dR)=\dim (\prf \dE_{[0,0; \delta]})=1.$ The dimension of the category $\prf \dE_{[0,1; \delta]}$ is not so easy to calculate, but using the methods of \cite[Sec. 6]{El}, one can show that it is also equal to 1.
We conjecture that the dimension of the category $\prf R_{\mathfrak{F}}$  is equal to 1, at least in the case $n=2.$

To see a difference between $\dE_{[0,0; \delta]}$ and $\dE_{[0,1; \delta]},$ we can try to consider dimension spectra.
Recall that the  {\sf dimension spectrum} $\sigma(\T)\subset\ZZ$ consists of
all $d\ge 0,$ for which there is an object $E\in \T$
with $\langle E\rangle_{d+1}=\T$ and $\langle E\rangle_{d}\ne\T$ (see \cite{Or09}).
It can be shown that $\sigma(\prf\dE_{[0,0; \delta]})=\{1,2\}.$ On the other hand, it is not so difficult to check that $\sigma(\prf\dE_{[0,1; \delta]})$ contains the integer 3.


The final remark is that the dimension of the category $\prf R_{\mathfrak{F}}$ does not exceed  3 for any $\mathfrak{F},$ since the nilpotency index is equal to 4, but at the same time, for any arbitrarily large odd number, there exists $\mathfrak{F}$ such that the spectrum of the algebra $R_{\mathfrak{F}}$ contains this number, since the spectrum contains the global dimension.


\end{document}